\documentclass[11pt]{article}
\usepackage[top=1in, bottom=1in, left=1in,right=1in]{geometry}
\usepackage{amsmath}
\usepackage{amssymb}
\usepackage{color}
\usepackage{amsthm}
\usepackage{mathrsfs}   
\usepackage{theoremref}
\usepackage{stix}
\usepackage{hyperref}
\usepackage{graphicx}

\usepackage[all]{xy}
\usepackage{tikz-cd}

    \newcommand\lmod[2]{
        \mathchoice
            {
                \text{\raise1ex\hbox{$#1$}\Big/\lower1ex\hbox{$#2$}}%
            }
            {
                #1\,/\,#2
            }
            {
                #1\,/\,#2
            }
            {
                #1\,/\,#2
            }
    }

\newtheorem{defn}{Definition}[section]
\newtheorem{thm}[defn]{Theorem}
\newtheorem*{thm*}{Theorem}
\newtheorem{lem}[defn]{Lemma}  

\newtheorem{cor}[defn]{Corollary}
\newtheorem{conj}[defn]{Conjecture}
\theoremstyle{definition}
\newtheorem*{rmk}{Remark}
\theoremstyle{definition}

  \newcommand{\F}{\mathbb{F}}

\newcommand{\Gal}{\operatorname{Gal}}
\newcommand{\disc}{\operatorname{disc}}
\newcommand{\Frob}{\operatorname{Frob}}

\title{Counting pairs of conics over finite fields that satisfy the Poncelet $n$-gon condition}
\author{Tianhao Wang \\ \href{mailto:tianhw11@uci.edu}{tianhw11@uci.edu}}
\date{\today}

\begin{document}
\maketitle

\begin{abstract}
An ordered pair of smooth conics satisfies the Poncelet triangle condition if there is a triangle inscribed in the first conic and circumscribed in the second conic. Over a finite field $\mathbb{F}_q$ with characteristic greater than $3$, Chipalkatti showed that the density of pairs of smooth conics  satisfying the Poncelet triangle condition is $\frac{1}{q}+O(q^{-2})$. We improve this result, showing that the density is exactly $\frac{q-1}{q^2-q+1}$.

We consider the problem of determining the density of pairs of conics satisfying the Poncelet $n$-gon condition for larger $n$.  We prove a corrected version of a conjecture of Chipalkatti, showing that the proportion of pairs of smooth conics satisfying the Poncelet tetragon condition is $\frac{1}{q} + O(q^{-{3/2}})$.  We show that when $n$ is an odd integer coprime to $q$, the density of pairs of smooth conics satisfying this condition is $\frac{d(n)-1}{q}+O(q^{-3/2})$, where $d(n)$ is the number of divisors of $n$. More generally, we conjecture that the density of pairs of conics satisfying the Poncelet $n$-gon condition is $d'(n)/q$ in general, where $d'(n)$ is the number of divisors of $n$ not equal to $1$ or $2$.  Our argument involves analyzing the $n$-torsion points on a certain elliptic curve over the function field $K = \mathbb{F}_q(\lambda)$.

\end{abstract}


\section{Introduction}
\subsection{Poncelet triangles and Poncelet $n$-gons}
Let $(\mathcal{A}, \mathcal{B})$ be an ordered pair of smooth plane conics defined over a field $k$ with transversal intersection.  We say that $(\mathcal{A}, \mathcal{B})$ \textbf{satisfies the Poncelet triangle condition} if there is a triangle defined over $\overline{k}$ that is inscribed in $\mathcal{A}$ and circumscribed in $\mathcal{B}$. 

\begin{figure}[h]
    \centering
    \includegraphics[width=5cm]{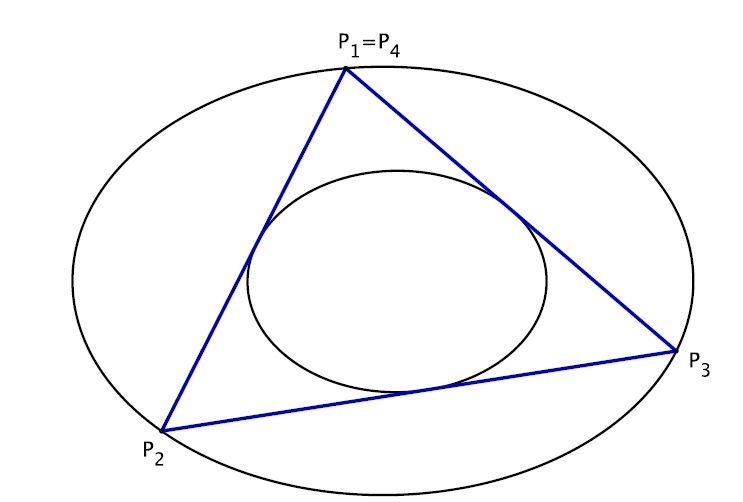}
    \caption{An example of a Poncelet triangle. This image is taken from \cite{Chi}.}
\end{figure}

We are interested in the density of pairs of smooth conics satisfying the Poncelet triangle condition when $k = \mathbb{F}_q$. We introduce the \textbf{Poncelet construction} to explain what we mean by "inscribed in $\mathcal{A}$ and circumscribed in $\mathcal{B}$" in an arbitrary field $k$. 

Suppose that the conics $\mathcal{A}$ and $\mathcal{B}$ intersect transversally. Let $\mathcal{B}^*$ be the dual conic of $\mathcal{B}$,  consisting of tangent lines to $\mathcal{B}$. We define $E\subset\mathcal{A}\times \mathcal{B}^*$ as the set of all $(P, \xi)$ where $P\in \xi$.  The idea of the Poncelet construction is that we start from $(P, \xi)\in E$.  The tangent line $\xi$ intersects $\mathcal{A}$ at another point $P'$, and we let $\xi'$ be the other tangent line to $\mathcal{B}$ passing through $P'$. This defines an operation 
\begin{equation*}
     j: (P, \xi) \mapsto (P', \xi').
\end{equation*}

\begin{figure}[h]
    \centering
    \includegraphics[width=7cm]{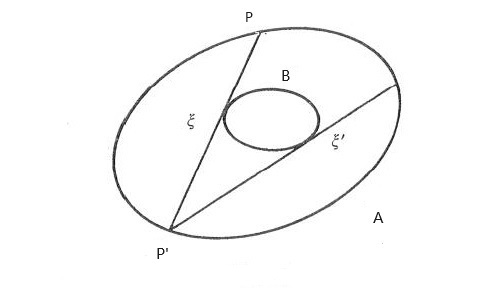}
    \caption{The Poncelet construction. This image is taken from \cite{GH}.}
\end{figure}

The pair $(\mathcal{A}, \mathcal{B})$ \textbf{satisfies the Poncelet triangle condition} if there is $(P,\xi)\in E$ such that $j^3(P, \xi) = (P, \xi)$. We also call $(P, \xi)$ the initial data for the construction. More generally, we say that $(\mathcal{A}, \mathcal{B})$ \textbf{satisfies the Poncelet $n$-gon condition}  if there is $(P,\xi)\in E$ such that $j^n(P,\xi) = (P, \xi)$. 

We observe that $E$ is a ramified cover of $\mathcal{A}$ via the natural projection. This is a degree $2$ covering with $4$ ramification points, each of ramification index $2$, corresponding to the $4$ intersection points of $\mathcal{A}$ and $\mathcal{B}$. Hence, by the Riemann-Hurwitz formula, we know that $E$ has genus $1$, and so it is an elliptic curve. Griffiths and Harris showed that $E$ is birationally equivalent to the elliptic curve with equation 
$$y^2 = \det(xA+B)$$
where $A,B$ are the matrix representation of the conics $\mathcal{A}, \mathcal{B}$ \cite{GH} . They further give a criterion for $(\mathcal{A}, \mathcal{B})$ to satisfy the Poncelet $n$-gon condition in terms of this elliptic curve.  

\begin{thm}[Griffiths-Harris]\label{Thm:GH}\par
Let $(\mathcal{A}, \mathcal{B})$ be an ordered pair of smooth plane conics with transversal intersection, and $(A,B)$ be their matrix representations. Then the pair $(\mathcal{A}, \mathcal{B})$ satisfies the Poncelet $n$-gon condition if and only if the point $(0, \sqrt{\det(B)})$ is an $n$-torsion point of the elliptic curve $y^2 = \det(xA+B)$. 
\end{thm}

One consequence of this theorem is that the initial data $(P, \xi)$ for the Poncelet construction does not matter. If there is a Poncelet $n$-gon with some initial data, then any initial data will give a Poncelet $n$-gon. This is the statement of Poncelet's theorem.

\subsection{Main results}
The first density result for pairs of conics satisfying the Poncelet triangle condition was obtained by Chipalkatti \cite{Chi}.

\begin{thm}\cite[Theorem 2.1]{Chi}\label{Thm:Chi} \par
    Let $\mathbb{F}_q$ be a finite field with characteristic greater than $3$. Let $\Psi$ be the set of ordered pairs of smooth conics with transversal intersection, and $\Gamma_3$ be the subset consisting of those satisfying the Poncelet triangle condition. Then we have 
    $$\frac{q-16}{q(q+1)}\leq \frac{|\Gamma_3|}{|\Psi|}\leq \frac{q+5}{(q-2)(q-3)}.$$
    In particular $\frac{|\Gamma_3|}{|\Psi|} = \frac{1}{q}+O\left(\frac{1}{q^2}\right)$. 
    If $\mathbb{F}_q$ has characteristic $3$, then $\frac{|\Gamma_3|}{|\Psi|} = \frac{2}{q}+O\left(\frac{1}{q^2}\right)$. 
\end{thm}

The idea in Chipalkatti's proof is to first classify pairs of smooth conics with transversal intersection up to projective automorphisms defined over $\mathbb{F}_q$.  Such a class is determined by the field of definition of the four intersection points of the two conics, and there are exactly $5$ cases to consider.  Chipalkatti gives the exact count of pairs of smooth conics satisfying the Poncelet triangle condition in the case where each of the four intersection points is defined over $\mathbb{F}_q$ and gives bounds for the density when the conics have the other possible intersection types.  The density turns out to be asymptotically $1/q$ in each of the classes.

We follow the strategy of Chipalkatti and obtain the precise count for $|\Gamma_3|$. 
\begin{thm}
If $\mathbb{F}_q$ has characteristic greater than $3$, then
    $$|\Gamma_3| =  (q^5-q^2)(q + 1)q(q - 1)^2,$$
    and therefore
    $$\frac{|\Gamma_3|}{|\Psi|} = \frac{q-1}{q^2-q+1}.$$
\end{thm}

Let $\Gamma_n$ be the subset of $\Psi$ consisting of pairs of smooth conics satisfying the Poncelet $n$-gon condition. Chipalkatti also made a conjecture based on computational results that $|\Gamma_n|/|\Psi|\sim \mu_n/q$ where 
$$\mu_4 = 3, \mu_5 = 1, \mu_6 = 4, \mu_7 = 1, \mu_8 = 6, \mu_9 = 2$$
However, it seems that the data was obtained from computing the density in a fixed class of smooth conic pairs. We prove that for $n = 4$, the density of pairs of smooth conics satisfying the Poncelet $n$-gon condition is different in different classes. Averaging over the possible intersection types, we prove that $\mu_4 = 1$, which is different from the conjectured value.  We again follow the strategy of Chipalkatti, where we also use the Hasse-Weil estimate for the number of $\mathbb{F}_q$-rational points of a curve. However, it is hard to proceed beyond $n=4$ because of the complexity of the equations involved. 

In Section \ref{Sec:ngon}, we return to Theorem \ref{Thm:GH} and describe the connection between the density of pairs of smooth conics with a given intersection type satisfying the Poncelet $n$-gon condition and the $n$-torsion points of an elliptic curve over $\mathbb{F}_q(\lambda)$.
\begin{thm}(Main result on the $n$-torsion points of a family of elliptic curves) \label{thm:main1} \par
    Let $\mathbb{F}_q$ be a finite field with characteristic greater than $3$. Let $E$ be a non-isotrivial elliptic curve defined over the function field $K = \mathbb{F}_q(\lambda)$. We denote the reduction of $E$ at a degree $1$ place $\lambda_0\in\mathbb{P}^1(\mathbb{F}_q)$ by $E_{\lambda_0}$. Let $r(n, \lambda_0)$ be the number of roots of the $n$-th division polynomial $\Lambda_n(x, \lambda_0)$ of $E_{\lambda_0}$ defined over $\mathbb{F}_q$. 
    
    Then, for all but finitely many odd numbers $n$ coprime to $q$, we have
    $$\sum_{\lambda_0}r(n,\lambda_0) = (d(n)-1)q+O(\sqrt{q})$$
    where the sum is over all $\lambda_0\in\mathbb{P}^1(\mathbb{F}_q)$ such that $E$ has a good reduction at $\lambda_0$. 

\end{thm}

We prove this theorem by analyzing the \textbf{\(n\)-torsion point field}, \(K(E[n])\), which is the field extension of \(K\) obtained by adjoining the \(x\)- and \(y\)-coordinates of all \(n\)-torsion points on \(E\). We demonstrate that the number of roots of \(\Lambda_n(x, \lambda_0)\) is governed by the conjugacy class of the Frobenius element at the place \(\lambda_0\). Finally, we apply the Chebotarev density theorem to complete the proof.

Theorem \ref{thm:main1} holds for an integer $n$ if $\operatorname{SL}_2(\mathbb{Z}/n\mathbb{Z})\subset \Gal(K(E[n]/K)$. We will be interested in 3 particular elliptic curves that correspond to different intersection types of pairs of smooth conics. For those $3$ particular elliptic curves, we will show that Theorem \ref{thm:main1} holds for all odd integers $n$ coprime to $q$

With the above theorem, we are able to obtain the density of smooth conic pairs satisfying the Poncelet $n$-gon condition. 
\begin{thm}\label{thm:n-gon} (Density of smooth conic pairs satisfying the Poncelet $n$-gon condition) \par
    Let $\mathbb{F}_q$ be a finite field with characteristic greater than $3$, and $n$ be an odd integer coprime to $q$. We have
    $$\frac{|\Gamma_n|}{|\Psi|} = \frac{d(n)-1}{q}+O(q^{-3/2}). $$
\end{thm}


\section{Poncelet triangles and tetragons}
\subsection{Pencils of conics and Cayley's criterion}
We first introduce Cayley's criterion on the existence of a Poncelet triangle, which is used in the proof of Theorem \ref{Thm:Chi}.  Let $k$ be a field with characteristic greater than $2$. A projective plane conic $\mathcal{A}$ over $k$ is defined by the equation 
$$ax^2+bxy+cxz+dy^2+eyz+fz^2 = 0.$$
We can associate a symmetric matrix to this conic:
$$A = \begin{bmatrix}a & b/2 & c/2 \\ b/2 & d & e/2 \\ c/2 & e/2 & f \end{bmatrix}. $$
The conic $\mathcal{A}$ is smooth if and only if $\det A \neq 0$. 
Any smooth conic $\mathcal{A}$ is isomorphic over $\overline{k}$ to the conic $\mathcal{I}: x^2+y^2+z^2 = 0$. 

Let $(\mathcal{A}, \mathcal{B})$ be an ordered pair of smooth conics with transversal intersection, and $E: y^2 = \det(xA+B)$ be the elliptic curve from Theorem \ref{Thm:GH}. We consider the following formal Maclaurin expansion
$$\sqrt{\det(xA+B)} = H_0+H_1x+H_2x^2+\cdots.$$
\begin{thm}(Cayley's criterion)\label{Thm:Cayley} \par
A pair of smooth projective plane conics $(\mathcal{A},\mathcal{B})$ with transversal intersection satisfy the Poncelet triangle condition if and only if $H_2 = 0$.
\end{thm}
For a proof, see \cite{GH}.  This is also the criterion for the point $P\in E$ with $x(P) = 0$ to be a $3$-torsion point of $E$. There is a generalization of the Cayley's criterion which determines when satisfy the Poncelet $n$-gon condition.  We define the \textbf{Hankel determinant}
$$f_{2m+1} = \det\begin{bmatrix}H_2 & H_3 &\cdots & H_{m+1} \\ H_3 & H_4 & \cdots & H_{m+2} \\ \vdots & \vdots &  & \vdots \\ H_{m+1} & H_{m+2} & \cdots & H_{2m}\end{bmatrix}\;\;\;f_{2m} =  \det\begin{bmatrix}H_3 & H_4 & \cdots & H_{m+1} \\ H_4 & H_5 & \cdots & H_{m+2} \\ \vdots & \vdots & & \vdots \\ H_{m+1} & H_{m+2} & \cdots & H_{2m}\end{bmatrix}.$$
Then the pair $(\mathcal{A}, \mathcal{B})$ satisfy the Poncelet $n$-gon if and only $f_n = 0$ \cite{GH}. In particular, we see that $f_3 = H_2$ and $f_4 = H_3$. 

The other ingredient in the proof of Theorem \ref{Thm:Chi} is the classification of pairs of smooth conics up to projective automorphism. We introduce the notion of a \textbf{pencil of conics} in order to state the classification.   Let $\mathcal{A}, \mathcal{B}$ be distinct projective plane conics, and $F(x,y,z), G(x,y,z)$ be their defining equations. The one-parameter family of conics $\{\eta F+G = 0: \eta\in\mathbb{P}^1(k)\}$ is called the \textbf{pencil of conics} generated by $\mathcal{A}, \mathcal{B}$, where $\eta = \infty$ gives the conic $\{F = 0\}$. We see that any pair of distinct conics in this pencil will share the same set of intersection points. If the generators $\mathcal{A}, \mathcal{B}$ intersect transversally at $4$ points in $\overline{k}$, then these $4$ points will determine the pencil uniquely, in the sense that the pencil consists of all conics passing through those $4$ distinct points.  In \cite[Section 3]{Chi}, Chipalkatti describes Dickson's classification of pencils of conics, which shows that there are only $5$ projective equivalence classes of pencils of smooth conics with transversal intersection. If two conics defined over $\mathbb{F}_q$ has transversal intersection, then their intersection points are permuted by the Frobenius automorphism, and the projective equivalence class of the pencil is determined by the cycle structure of this permutation, which gives a partition of the integer $4$. We refer to the projective equivalence class of the pencil by this cycle structure, so for example we write $(1,1,1,1)$ to denote the class of the pencil where the two generators intersect at $4$ $\mathbb{F}_q$-points, and write $(2,1,1)$ to denote the class where the generators intersect at a a pair of conjugate $\mathbb{F}_{q^2}$-points, and two other $\mathbb{F}_q$-points.

The following table lists each of the projective equivalence classes of pencils and also the number of pencils in each class.  The entries of this table can be determined by counting sets of four points in $\mathbb{P}^2(\overline{\F}_q)$ with no three points in a line that are fixed by the Frobenius automorphism.  We omit the proof since this kind of count is standard.  For a more complicated analogue of this results for collections of five points, and collections of six points, see Table 1 and Table 2 in \cite{DasOConnor}. 
\begin{center}
    \begin{tabular}{|c|c|}
    \hline
    \textbf{Pencil Intersection Type} & \textbf{Number of Pencils of This Type} \\ \hline
    $(1,1,1,1)$ &  $\frac{1}{24}(q^8 - q^6 - q^5 + q^3)$ \\ \hline 
    $(2,1,1)$ & $\frac{1}{4}(q^8 - q^6 - q^5 + q^3)$ \\ \hline
    $(2,2)$ & $\frac{1}{8}(q^8 - q^6 - q^5 + q^3)$ \\ \hline
    $(3,1)$ & $\frac{1}{3}(q^8 - q^6 - q^5 + q^3)$ \\ \hline
    $(4)$ & $\frac{1}{4}(q^8 - q^6 - q^5 + q^3)$\\
    \hline
\end{tabular}
\end{center}

There are also other equivalence classes of pencils of conics where the common intersection may not be $4$ distinct points, and there are extra classes when $k$ has characteristic $2$. For a complete classification of the classes of pencil of conics, we refer the reader to Table 7.7 of Hirschfeld's \emph{Projective Geometries over Finite Fields} \cite{Hirschfeld}.

\subsection{Number of pairs of smooth conics satisfying the Poncelet triangle condition}\label{triangle}
We follow Chipalkatti's strategy for the proof of Theorem \ref{Thm:Chi}.  For each of the equivalence classes of pencils of conics described in the preivous section, we choose an explicit pair of conics generating the pencil and use the equations defining these conics to compute $H_2$ from Theorem \ref{Thm:Cayley}.  We count the number of $\F_q$-points of the resulting curve $\{H_2 = 0\}$.  The computation for the pencil class with intersection type $(1,1,1,1)$ is done in \cite{Chi}. We reproduce it here, since we will use the same strategy for the other classes. 

Let $\mathbb{F}_q$ be a finite field with characteristic greater $3$, and $\{F,G\}$ be a generator of a pencil of conics. We write $C_r = rF+G$ with $r\in\mathbb{P}^1(\mathbb{F}_q)$ to denote a conic in the pencil, and $H_2(r,s)$ for the $H_2$ associated to the pair of conics $(C_r, C_s)$. We will see that $H_2(r,s)$ is a quadratic function in $r$, and compute its discriminant $(\disc H_2)(s)$, which is a function of $s$. Then, we count the number of $s_0\in\mathbb{F}_q$ such that $H_2(r,s_0)$ has $0,1,2$ $\mathbb{F}_q$-roots respectively. 

We also observe that if $A$ is a matrix representation of a smooth conic $\mathcal{A}$, then for the smooth conic pair $(\mathcal{A}, \mathcal{A})$, we can compute $H_2 = 3\det A \neq 0$. Therefore, if we get $H_2(r,s) = 0$ in a pencil of conics, it is a guarantee that $C_r, C_s$ are distinct conics and they intersect transversally.

We will also see that the number of pairs of smooth conics satisfying the Poncelet triangle condition in a pencil of type $(1,1,1,1)$  is the same as the number of such conics in a pencil of type $(2,2)$.  In Section \ref{Sec:ngon} we will relate this observation to the $3$-torsion points of certain elliptic curves.  The main point is that there are $3$ singular conics defined over $\mathbb{F}_q$ in each of the pencils. We note that a corresponding statement about counting pairs of conics satisfying the Poncelet triangle condition holds for pencils of type $(4)$ and pencils of type $(2,1,1)$.  Each of these pencils contains exactly $1$ singular conic defined over $\mathbb{F}_q$.

We also note that there will be different behaviors when $q\equiv 1 \pmod 3$ and when $q\equiv -1 \pmod 3$. We will see the similar behavior again when counting pairs of smooth conics satisfying the Poncelet $n$-gon condition.

\subsubsection{Intersection type (1,1,1,1)}

\begin{itemize}
    \item Generators of the pencil: $F = xy$, $G = z^2+yz+xz$.
    \item Singular members: $\eta_1 = 0$, $\eta_2 = 1$, $\eta_3 = \infty$. 
    \item After clearing denominators, $H_2(r, s) = r^2+(6s^2-4s^3-4s)r+s^4$.
    \item $(\disc H_2)(s) = 16(s^2-s+1)(s(s-1))^2$. 
    \item Density of pairs of smooth conics satisfying the Poncelet triangle condition: $\frac{q-5}{(q-2)(q-3)}$.
 \end{itemize}
  
We now justify the count for pairs of smooth conics in this pencil that satisfy the Poncelet triangle condition.  The claim about the singular members of this pencil is clear.  There are $q+1$ choices for $r$, and only $q-2$ of them define smooth conics. Hence, the number of pairs of smooth conics with transversal intersection in this pencil is $(q-2)(q-3)$.

We first remove solutions that correspond to a pair of conics in which at least one member is singular.  When $s=0,1$, we get $(\disc H_2)(s) = 0$, and there is one value of $r$ corresponding to each of these two values of $s$, which are $r=0,1$ respectively. We also see that when $r=0,1$, there is exactly one value of $s$ such that $H_2(r,s) = 0$, which are $s=0,1$ respectively. 

    Now, we assume that $s\neq 0,1$. We know that if $H_2(r,s) = 0$, we will also have $r\neq 0,1$. Hence, if $s\neq 0,1$ and $H_2(r,s) = 0$, we know that both $C_r$ and $C_s$ are smooth. 
    
    Since $s \neq 0,1$, the square factor $(s(s-1))^2$ of $(\disc H_2)(s)$ is always a non-zero square. We follow the notation in \cite{Chi}, we denote the non-square part of $(\disc H_2)(s)$ as $f(s) = s^2-s+1 = (s-1/2)^2+3/4$. Therefore we can determine the number of roots of $(\disc H_2)(s)$ by determining whether $f(s)$ is zero, a non-zero square, or a non-zero non-square.
    
    If $f(s) = y^2$ for some $y\in\mathbb{F}_q$, then we have 
    $$(y-s+1/2)(y+s-1/2) = \frac{3}{4}.$$
    Let $a = y -s + 1/2$, which implies $y+s-1/2 = \frac{3}{4a}$. Solving for $s$ and $y$ in terms of $a$, we get 
    $$s = \frac{3+4a-4a^2}{8a}, \ \ \ y = \frac{4a^2+3}{8a}.$$
    This allows us to parameterize the values of $s$ such that $f(s)$ is a square in $\mathbb{F}_q$. 

    We see that $a$ and $-3/(4a)$ give the same value for $s$ in the above formula. Hence, each $s$ value with $f(s)$ being a square has two distinct $a$ corresponding to it except when $a = -3/(4a)$. In this case, $a = \pm \frac{\sqrt{-3}}{2},\ s = \frac{1}{2}\pm \frac{\sqrt{-3}}{2}$, and $f(s) = 0$. We separate the discussion into two cases based on whether $-3$ is a square in $\mathbb{F}_q$. 

\textbf{Case 1:} If $q\equiv -1 \pmod 3$, then $-3$ is not a square in $\mathbb{F}_q$. We know that $f(s)$ has no roots defined over $\F_q$, and each $s$ such that $f(s)$ is a non-zero square has two values of $a$ corresponding to it. Hence, the $q-1$ choices for $a$ give $\frac{q-1}{2}$ choices for $s$ such that $f(s)$ is a non-zero square. After excluding $s = 0, 1$ where $f(s)$ is also a square, we have $\frac{q-5}{2}$ values of $s$ such that $(\disc H_2)(s)$ is a non-zero square. There are $2$ values of $r$ corresponding to each of these values of $s$. In total, we get $q-5$ pairs of smooth conics $(C_r, C_s)$ in this pencil satisfying the Poncelet triangle condition. 

\textbf{Case 2:} If $q\equiv 1 \pmod 3$, then $-3$ is a square in $\mathbb{F}_q$. In this case $f(s)$ has two roots defined over $\F_q$.  The values of $a$ corresponding to these roots are $a = \pm\sqrt{-3}/2$. Hence the $q-3$ choices for $a$ will give $\frac{q-3}{2}$ values of $s$ such that $f(s)$ is a non-zero square. By excluding $s = 0,1$ as we did above, there are $2$ values $s$ such that $(\disc H_2)(s) = 0$, and $\frac{q-7}{2}$ values of $s$ such that $(\disc H_2)(s)$ is a non-zero square. Hence, there are in total $2 + 2\cdot \frac{q-7}{2} = q-5$ pairs of smooth conics $(C_r, C_s)$ in this pencil satisfying the Poncelet triangle condition.

\subsubsection{Intersection type (2,1,1)}
\begin{itemize}
    \item Generators of the pencil: $F = xy$, $G = y^2+yz+xz+ez^2$ where $T^2+T+e$ irreducible over $\mathbb{F}_q$.
    \item Singular members: $\eta = \infty$.
    \item After clearing denominators, $H_2(r,s) = r^2(4e-1)+r(4s^3e^2-6s^2e-4se+4s-2)+(-s^4e^2+6s^2e-4s+3)$.
    \item $(\disc H_2)(s) = 16(s^2e^2 - se - 3e + 1)(s^2e - s + 1)^2$.
    \item Density of pairs of smooth conics satisfying the Poncelet triangle condition: $\frac{q-1}{q(q-1)} = \frac{1}{q}$.
\end{itemize}

We justify this last formula following the same strategy as in the previous case.  There are $q$ choices for $r$ so that $C_r$ is a smooth conic, and hence, there are in total $q(q-1)$ pairs of smooth conics in the pencil. 
    
    Since $T^2+T+e$ is irreducible over $\mathbb{F}_q$, we know that $1-4e$ is not a square in $\mathbb{F}_q$. In particular, it does not equal  zero, and hence, $H_2(r,s)$ is a quadratic polynomial in $r$.
        
    Consider the square factor $(s^2e - s + 1)^2$ of $(\disc H_2)(s)$.  Note that $s^2e-s+1$ has no roots over $\mathbb{F}_q$ since its discriminant is $1-4e$, which cannot be a square by assumption. Therefore, we can determine the number of roots of $(\disc H_2)(s)$ by determining whether $s^2e^2-se-3e+1$ is zero, a non-zero square, or a non-zero non-square.  This is equivalent to determining whether $s^2-se^{-1}-3e^{-1}+e^{-2}$ is zero, a non-zero square, or a non-zero non-square.
            
    We follow the same strategy as in previous case. We write $s^2-se^{-1}-3e^{-1}+e^{-2} = (s-b)^2+c$ where $b = \frac{e^{-1}}{2}$ and $c = \frac{3(1-4e)}{4e^2}\neq 0$. If $(s-b)^2+c = y^2$ for some $y\in \mathbb{F}_q$, then we have $(y-s+b)(y+s-b) = c$.  Let $a = y-s+b$, which is non-zero since $c\neq 0$.  This implies $y+s-b = \frac{c}{a}$.  By solving for $s$ in terms of $a$, we get 
        $$s = \frac{c+2ba-a^2}{2a}\;\;\;\;\;\;\;a\in\mathbb{F}_q\setminus \{0\}$$
        as a parametrization of $s$ in terms of $a$.  We note that $a, -c/a$ will give the same value for $s$. 

\textbf{Case 1:} If $q\equiv -1 \pmod 3$, then $-3$ is not a square in $\mathbb{F}_q$, and so $-c = \frac{-3(1-4e)}{4e^2}$ is a square in $\mathbb{F}_q$. Then $(s-b)^2+c$ has two roots over $\mathbb{F}_q$ which are $s = b\pm(\sqrt{-c})$.  These roots occur when $a = -c/a = \pm\sqrt{-c}$. For each of the other $q-3$ choices of $a$, there is another value of $a$ giving the same value of $s$, and each of these $s$ are such that $(s-b)^2+c$ is a non-zero square. In summary, we have $2$ values of $s$ such that $(\disc H_2)(s) = 0$, which means that each of them has $1$ value of $r$ such that $H_2(r,s) = 0$. There are also $\frac{q-3}{2}$ values of $s$ such that $(\disc H_2)(s)$ is a non-zero square, which means that for each of these values there are $2$ values of $r$ such that $H_2(r,s) = 0$. Then, we have $2+\frac{q-3}{2}\cdot 2 = q-1$ pairs of smooth conics in this pencil satisfying the Poncelet triangle condition. 
        
\textbf{Case 2:} If $q\equiv 1 \pmod 3$, then $-c$ is not a square in $\mathbb{F}_q$. We know that $(s-b)^2+c$ has no roots over $\mathbb{F}_q$, and there are $\frac{q-1}{2}$ values of $s$ such that it is a non-zero square. In total, there are $\frac{q-1}{2}\cdot 2 = q-1$ pairs of smooth conics in this pencil satisfying the Poncelet triangle condition.  

\subsubsection{Intersection type (2,2)}
\begin{itemize}
    \item Generators: $F = xy$, $G = e_1x^2+e_2y^2+xz+yz+z^2$ where $T^2+T+e_1$, $T^2+T+e_2$ are irreducible over $\mathbb{F}_q$.
    \item Singular members: $\eta_3 = \infty$, $\eta_1 = \frac{1+\sqrt{(1-4e_1)(1-4e_2)}}{2}$, $\eta_2 = \frac{1-\sqrt{(1-4e_1)(1-4e_2)}}{2}$.
    \item After clearing denominators, $H_2(r,s) =r^2(-16e_1e_2+4e_1+4e_2-1)+r(16se_1e_2+4s^3-6s^2-4se_1-4se_2+8e_1e_2+4s-2e_1-2e_2)+(-s^4-24s^2e_1e_2+48e_1^2e_2^2+6s^2e_1+6s^2e_2+16se_1e_2-24e_1^2e_2-24e_1e_2^2-4se_1-4se_2+3e_1^2+3e_2^2+6e_1e_2)$.
    \item $(\disc H_2)(s) = 16(s^2 + 12e_1e_2 - s - 3e_1 - 3e_2 + 1)(-s^2 + 4e_1e_2 + s - e_1 - e_2)^2$.
    \item Density of pairs of smooth conics satisfying the Poncelet triangle condition: $\frac{q-5}{(q-2)(q-3)}$.
\end{itemize}


We now justify this last formula. Since $T^2+T+e_1$ and $T^2+T+e_2$ are irreducible over $\mathbb{F}_q$, we know that $1-4e_1$ and $1-4e_2$ are not squares in $\mathbb{F}_q$. Therefore, we know that $(1-4e_1)(1-4e_2)$ is a square in $\mathbb{F}_q$, and $\eta_1, \eta_2\in\mathbb{F}_q$. Then, there are $q-2$ choices of $r$ so that $C_r$ is a smooth conic, and $(q-2)(q-3)$ pairs of smooth conics in this pencil.

Since $1-4e_1, 1-4e_2$ are not squares over $\mathbb{F}_q$, neither of them can be zero. We see that the coefficient for $r^2$ in $H_2(r,s)$ is $(-16e_1e_2+4e_1+4e_2-1) = -(4e_1-1)(4e_2-1)\neq 0$. Hence, $H_2(r,s)$ is a quadratic polynomial in $r$. 
    
We consider the square factor $(-s^2 + 4e_1e_2 + s - e_1 - e_2)$ in $(\disc H_2)(s)$. It has two roots $\eta_1, \eta_2$. Hence, for $s\neq \eta_1, \eta_2$, we know that $\disc H_2$ is zero or a non-zero square square if and only if $s^2 + 12e_1e_2 - s - 3e_1 - 3e_2 + 1$ is zero or a non-zero square. 

We also need to remove solutions to $H_2(r,s) = 0$ where at least one of the conic is singular. When $s = \eta_1, \eta_2$, we compute that $\disc H_2 = 0$, and there is one value for $r$ corresponding to each of these two values of $s$. We also see that $H_2(\eta_1,\eta_1) = 0$, $H_2(\eta_2, \eta_2) = 0$, $H_2(\eta_1, s) = (s-\eta_1)^4$ and $H_2(\eta_2, s) = (s-\eta_2)^4$. 

Therefore, we know that if $H_2(r,s) = 0$ and $s\neq \eta_1, \eta_2$, then both of the conics $C_r, C_s$ are smooth. We also know that $(\disc H_2)(s)$ is zero or a non-zero square if and only if $s^2 + 12e_1e_2 - s - 3e_1 - 3e_2 + 1$ is a zero or a non-zero square.
        
We follow the same parameterization trick as previously. We write $s^2 + 12e_1e_2 - s - 3e_1 - 3e_2 + 1 = (s-b)^2+c$ where $b = \frac{1}{2}$ and $c = \frac{3}{4}(4e_1-1)(4e_2-1)\neq 0$. We solve that $(s-b)^2+c$ is a square if 
$$s = \frac{c+2ba-a^2}{2a}\;\;\;\;\;\;\;a\in\mathbb{F}_q\setminus \{0\}.$$

We note that $a, -c/a$ will give same value for $s$. If $a = -c/a$, then $a = \pm\sqrt{-c}$. In this case, we would have $s = b\pm\sqrt{-c}$. Those are the only two values of $s$ such that $(s-b)^2+c$ is a square (zero), and has only one pre-image in $a$. All other values of $s$ such that $(s-b)^2+c$ is a square will have exactly two pre-images in $a$. 

We also notice that when $s = \eta_1, \eta_2$, we would have $(s-b)^2+c = (1-4e_1)(1-4e_2)$ is a non-zero square in $\mathbb{F}_q$. Therefore, when counting the number of $s$ such that $(s-b)^2+c$ is a non-zero square, we need to remove those two values of $s$ corresponding to singular conics.
        
\textbf{Case 1:} If $q\equiv -1 \pmod 3$, then $-3$ is not a square in $\mathbb{F}_q$, and so $-c = \frac{-3}{4}(4e_1-1)(4e_2-1)$ is also not a square in $\mathbb{F}_q$. We know that $(s-b)^2+c$ has no roots in $\mathbb{F}_q$, and the $q-1$ choices of $a$ will yield $\frac{q-1}{2}$ choices of $s$ such that $(s-b)^2+c$ is a non-zero square. After removing $s = \eta_1, \eta_2$, we have $\frac{q-5}{2}$ choices of $s$ such that $\text{disc}H_2$ is a non-zero square. In total, there are $\frac{q-5}{2}\cdot 2 = q-5$ pairs of smooth conics in this pencil satisfying the Poncelet triangle condition,

\textbf{Case 2:} If $q\equiv 1\pmod 3$, then $-c$ is a square in $\mathbb{F}_q$. Following the same argument we applied in the previous pencil, we know that there are $\frac{q-3}{2}$ values of $s$ such that $(s-b)^2+c$ is a non-zero square, and $2$ values of $s$ such that it is zero. After removing $s = \eta_1, \eta_2$ corresponding to the singular conics, we have in total $2\cdot (\frac{q-3}{2} - 2) + 2 = q-5$ pairs of smooth conics in this pencil satisfying the Poncelet triangle condition.

\subsubsection{Intersection type (4)}
\begin{itemize}
    \item Generators: $F = x^2-ay^2$, $G = z^2-by^2+2cxy$ where $\sqrt{a}\notin\mathbb{F}_q$, $\sqrt{b^2-4ac^2}\notin\mathbb{F}_q$. 
    \item Singular members: $\eta = \infty$
    \item After clearing denominators, $H_2(r,s) =r^2(4ac^2-b^2)+r(4s^3a^2+6s^2ab-4sac^2+4sb^2+2bc^2)+(-s^4a^2+6s^2ac^2+4sbc^2+3c^4)$.
    \item $(\disc H_2)(s) = 16(s^2a^2 + sab - 3ac^2 + b^2)(s^2a + sb + c^2)^2$
    \item Density of smooth pairs satisfying the Poncelet triangle condition: $\frac{q-1}{q(q-1)} = \frac{1}{q}$.
\end{itemize}

    The proof in this pencil is exactly the same as the proof in the pencil of intersection type $(2,1,1)$. We sketch the proof instead of writing all the details.
    \begin{enumerate}
        \item $H_2$ is always a degree $2$ polynomial in $r$ by assumption.
        \item The square term of $(\disc H_2)(s)$ has no roots over $\mathbb{F}_q$, and hence $\disc H_2$ is a zero or a non-zero square if and only if $s^2a^2 + sab - 3ac^2 + b^2$ is a zero or a non-zero square. 
        \item There is only one singular conic $\eta = \infty$ in this pencil. We follow the same parametrization technique that we applied previously. Then, we can show that there are $q-1$ pairs of smooth conics in this pencil that satisfy the Poncelet triangle condition. 
    \end{enumerate}
\subsubsection{Intersection type (3,1)}
\begin{itemize}
    \item Generators: $F = y^2-xz$, $G = x^2+by^2+cxy+yz$ where $T^3+bT^2+cT+1$ is irreducible over $\mathbb{F}_q$. 
    \item Singular members: None
    \item After clearing denominators, $H_2(r,s) =r^2(3s^4+4s^3b+6s^2c-c^2+12s+4b)+r(2s^4b+4s^3b^2-4s^3c+6s^2bc+4sc^2-18s^2-4sb+2c)+(-s^4b^2+4s^4c+12s^3+6s^2b+4sc+3)$.
    \item $H_2(r, \infty) = 3r^2 + 2rb - b^2 + 4c$ with discriminant $16(b^2-3c)$.
    \item $(\disc H_2)(s) = 16(s^2(b^2 - 3c) + s(bc-9) + c^2 - 3b)(s^3 + s^2b + sc + 1)^2$.
    \item Density of pairs of smooth conics satisfying the Poncelet triangle condition: $\frac{q+1}{(q+1)q} = \frac{1}{q}$.
\end{itemize}

The main difference from the earlier proofs lies in the absence of singular conics defined over \( \mathbb{F}_q \) in this pencil. Therefore, we must account for the case \( r, s = \infty \) when solving \( H_2(r,s) = 0 \). Another difference is that \( H_2(r,s) \) may have degree less than 2 in the variable \( r \). To handle this, we treat the case \( s = \infty \) separately. For \( s \neq \infty \), we homogenize \( H_2(r,s) \) in the variable \( r \), allowing us to count its roots in \( \mathbb{P}^1(\mathbb{F}_q) \) using the discriminant. 

We observe that the square factor in \( (\disc H_2)(s) \) is \( s^3 + s^2b + sc + 1 \), which has no roots in \( \mathbb{F}_q \) by assumption. Consequently, \( (\disc H_2)(s) \) is zero or a non-zero square if and only if \( s^2(b^2 - 3c) + s(bc - 9) + c^2 - 3b \) is zero or a non-zero square.

The discriminant of \( s^2(b^2 - 3c) + s(bc - 9) + c^2 - 3b \) is given by $\Delta = 3b^2c^2 - 54bc + 12c^3 + 12b^3 + 81$, which equals \(-3\) times the discriminant of \( T^3 + bT^2 + cT + 1 \). Since \( T^3 + bT^2 + cT + 1 \) is irreducible over \( \mathbb{F}_q \) by assumption, we know that \( \Delta \neq 0 \). Additionally, \( b^2 - 3c \) and \( bc - 9 \) cannot both be zero, as this would make $\Delta = 0$. Hence, \( (\disc H_2)(s) \) is either a quadratic or linear polynomial in \( s \), but never constant.

We now analyze three cases:

\textbf{Case 1}: If \( b^2 - 3c = 0 \), then \( (\disc H_2)(s) \) is linear. There are \( \frac{q-1}{2} \) values of \( s \) such that \( (\disc H_2)(s) \) is a non-zero square, and one value where \( (\disc H_2)(s) = 0 \). Thus, we obtain \( 2 \cdot \frac{q-1}{2} + 1 = q \) pairs \((r,s)\) such that $H_2(r,s) = 0$. For \( s = \infty \), the discriminant of \( H_2(r, \infty) \) is zero, so there is exactly one \( r_0 \in \mathbb{F}_q \) where \( H_2(r_0, \infty) = 0 \). In total, this gives \( q + 1 \) smooth conic pairs in this pencil.

\textbf{Case 2}: If \( b^2 - 3c \) is a non-zero square in \( \mathbb{F}_q \), then the discriminant of \( H_2(r, \infty) \) is a non-zero square, so there are two roots \( r \in \mathbb{F}_q \) satisfying \( H_2(r, \infty) = 0 \). For \( s \neq \infty \), using the same counting method as before, we find $\frac{q-1}{2}$ values of $s_0\in\mathbb{F}_q$ such that $(\disc H_2)(s)$ is a non-zero square. In total, we find $2 + 2\cdot\frac{q-1}{2} = q+1$ pairs of smooth conics in this pencil satisfying the Poncelet triangle condition. 

\textbf{Case 3:} If $b^2-3c$ is not a square in $\mathbb{F}_q$, then \( H_2(r, \infty) \) has no roots in \( \mathbb{F}_q \). For \( s \neq \infty \), we divide \( (\disc H_2)(s) \) by \( b^2 - 3c \), observing that \( (\disc H_2)(s) \) is zero or non-zero square if and only if \( (\disc H_2)'(s) = s^2 + s\frac{bc - 9}{b^2 - 3c} + \frac{c^2 - 3b}{b^2 - 3c} \) is zero or non-square. We follow the parametrization technique as previous. 

If \( (\disc H_2)'(s) \) has roots in \( \mathbb{F}_q \), there are $2$ values of \( s \) where \( (\disc H_2)'(s) = 0 \), and $\frac{q-3}{2}$ values of $s$ where $(\disc H_2)'(s)$ is a non-zero square. Therefore, there are $q-(\frac{q-3}{2}+2) = \frac{q-1}{2}$ values of $s$ such that $(\disc H_2)'(s)$ is a non-square. In total, this results in \( 2 + 2 \cdot \frac{q-1}{2} = q + 1 \) pairs of smooth conics in this pencil that satisfy the Poncelet triangle condition. 

If \( (\disc H_2)'(s) \) has no roots in $\mathbb{F}_q$, there are $\frac{q-1}{2}$ values of $s$ such that $(\disc H_2)'(s)$ is a non-zero square. Therefore, there are $(q - \frac{q-1}{2}) = \frac{q+1}{2} $ values of \( s \) where \( (\disc H_2)'(s) \) is a non-square. In total, we find $2\cdot \frac{q+1}{2} = q+1$ pairs of smooth conics in this pencil satisfying the Poncelet triangle condition, 

\subsection{Asymptotic density of pairs of smooth conics satisfying the Poncelet tetragon condition}\label{tetragon}

In [\cite{Chi}, Section 4], Chipalkatti conjectured that the asymptotic density of pairs of smooth conics satisfying the Poncelet tetragon condition is $3/q$ as $q\rightarrow\infty$, based on computational data. However, we observe that the computational data only applies to conic pairs in the pencil of intersection type $(1,1,1,1)$. In this section, we demonstrate that the true asymptotic density for the tetragon case is $1/q$.  

\begin{thm}
    Let \(\mathbb{F}_q\) be a finite field of characteristic greater than \(3\). Using the same notation as in Theorem \ref{Thm:Chi}, we have  
    \[
    \frac{|\Gamma_4|}{|\Psi|} = \frac{1}{q} + O(q^{-3/2}).
    \]  
    Furthermore, the asymptotic density corresponding to each pencil type is given in the following table:  
    \begin{center}
        \begin{tabular}{|c|c|c|}
            \hline
            \textbf{Pencil Intersection Type} & \textbf{Density of Pencils} & \textbf{Density of Conic Pairs Satisfying} \\
            \ & \ & \textbf{the Poncelet Tetragon Condition} \\ \hline
            \((1,1,1,1)\) & \(1/24\) & \(3/q + O(q^{-3/2})\) \\ \hline
            \((2,1,1)\) & \(1/4\) & \(1/q + O(q^{-3/2})\) \\ \hline
            \((2,2)\) & \(1/8\) & \(3/q + O(q^{-3/2})\) \\ \hline 
            \((3,1)\) & \(1/3\) & \(0\) \\ \hline 
            \((4)\) & \(1/4\) & \(1/q + O(q^{-3/2})\) \\ \hline
        \end{tabular}
    \end{center}
\end{thm}

\begin{proof}
By Cayley's criterion (Theorem \ref{Thm:Cayley}), a pair of smooth conics satisfies the Poncelet tetragon condition if and only if \(H_3 = 0\). We follow the approach from the previous section, computing the \(\mathbb{F}_q\)-roots of \(H_3(r,s)\) in each pencil. Using the Hasse-Weil estimate, we calculate the asymptotic number of \(\mathbb{F}_q\)-rational points on the curve \(H_3(r,s) = 0\) as \(q \to \infty\). The numerator of the density result in the table above corresponds to the number of geometrically irreducible components of \(H_3(r,s) = 0\) that is defined over $\mathbb{F}_q$. Additionally, we ignore the smoothness requirement of the conics when calculating the asymptotic density, as each pencil contains at most three singular conics.

\begin{itemize}
    \item  \textbf{Intersection type $(1,1,1,1)$}: 

    Using the same generators for the pencil as in Section \ref{triangle}, we factor \(H_3(r,s)\) over \(\mathbb{F}_q\) (after dropping denominators)
    \[
    H_3(r,s) = (-2rs + s^2 + r)(s^2 - r)(s^2 + r - 2s).
    \]

    Each factor is also irreducible over \(\overline{\mathbb{F}_q}\). Hence, \(H_3(r,s)\) has three geometrically irreducible components defined over \(\mathbb{F}_q\). By the Hasse-Weil estimate, \(H_3(r,s)\) has \(3q + O(\sqrt{q})\) \(\mathbb{F}_q\)-rational points. Thus, the density of conic pairs satisfying the Poncelet tetragon condition is \(\frac{3}{q} + O(q^{-3/2})\) in this pencil.

    \item \textbf{Intersection type $(2,1,1)$}: 

    Let the pencil generators be \(F = xy\) and \(G = y^2 + yz + xz + ez^2\), where \(T^2 + T + e\) is irreducible over \(\mathbb{F}_q\). After dropping denominators, \(H_3(r,s)\) factors as
    \[
    H_3(r,s) = (-2rse + s^2e + r - 1)(s^4e^2 - 2s^3e + 4r^2e - 8rse + 6s^2e - r^2 + 2rs - 2s + 1).
    \]
    The second term factors over $\overline{\mathbb{F}_q}$ as
    \[
    e^2(s^2 + Ar + Bs + e^{-1})(s^2 - Ar + (-2e^{-1} - B)s + e^{-1}),
    \] 
    where \(A = e^{-1}\sqrt{1 - 4e}\) and \(B = -e^{-1} - e^{-1}\sqrt{1 - 4e}\). Since \(T^2 + T + e\) is irreducible over \(\mathbb{F}_q\), \(\sqrt{1 - 4e} \in \mathbb{F}_{q^2} \setminus \mathbb{F}_q\), and thus \(A, B \in \mathbb{F}_{q^2} \setminus \mathbb{F}_q\).

    Therefore, \(H_3(r,s)\) has three irreducible components over \(\overline{\mathbb{F}_q}\), one of which is defined over \(\mathbb{F}_q\), contributing \(q + O(\sqrt{q})\) \(\mathbb{F}_q\)-rational points. The $\mathbb{F}_q$-rational points on the second \(\mathbb{F}_q\)-component arise from the intersection of its two \(\overline{\mathbb{F}_q}\)-components, contributing at most four points. Thus, \(H_3(r,s)\) has \(q + O(\sqrt{q})\) rational points, and the density of conic pairs satisfying the Poncelet tetragon condition is \(\frac{1}{q} + O(q^{-3/2})\) in this pencil.

    \item \textbf{Intersection type $(2,2)$}:
    
    We choose generators \(F = xy\) and \(G = e_1x^2 + e_2y^2 + xz + yz + z^2\), where \(T^2 + T + e_1\) and \(T^2 + T + e_2\) are irreducible over \(\mathbb{F}_q\). We factor \(H_3(r,s)\) over $\mathbb{F}_q$ as
    \[
    H_3(r,s) = (-2rs + s^2 + r + 4e_1e_2 - e_1 - e_2)(s^2 + Ar + Bs + C)(s^2 - Ar + (-2 - B)s + C),
    \]
    where \(A = \sqrt{(1 - 4e_1)(1 - 4e_2)}\), \(B = -1 - \sqrt{(1 - 4e_1)(1 - 4e_2)}\), and \(C = e_1 + e_2 - 4e_1e_2\). We have $A,B,C\in\mathbb{F}_q$, since \(T^2 + T + e_1\) and \(T^2 + T + e_2\) are irreducible over \(\mathbb{F}_q\). 

    Therefore, $H_3(r,s)$ has $3$ components over $\mathbb{F}_q$, and those components are also irreducible over $\overline{\mathbb{F}_q}$. It follows that the density of conic pairs satisfying the Poncelet tetragon condition is $3/q + O(q^{-3/2})$ in this pencil.

    \item \textbf{Intersection type $(4)$}:
    
    Let the generators be \(F = x^2 - ay^2\) and \(G = z^2 - by^2 + 2cxy\), where \(\sqrt{a} \notin \mathbb{F}_q\) and \(\sqrt{b^2 - 4ac^2} \notin \mathbb{F}_q\). Factoring \(H_3(r,s)\) over \(\overline{\mathbb{F}_q}\), we have
    \[
    H_3(r,s) = a^2(2rsa - s^2a + rb + c^2)(s^2 + Ar + Bs + a^{-1}c^2)(s^2 - Ar + (2a^{-1}b - B)s + a^{-1}c^2),
    \]
    where \(A = a^{-1}\sqrt{b^2 - 4ac^2}\) and \(B = a^{-1}b - a^{-1}\sqrt{b^2 - 4ac^2}\). Since \(A, B \notin \mathbb{F}_q\), \(H_3(r,s)\) has one \(\mathbb{F}_q\)-component and two conjugate \(\mathbb{F}_{q^2}\)-components. By a similar argument as in type \((2,1,1)\), the density is \(\frac{1}{q} + O(q^{-3/2})\).
    
    \item \textbf{Intersection type $(3,1)$}:

    We pick generators $F = y^2-xz$ and $G = x^2+by^2+cxy+yz$ for this pencil, where $T^3+bT^2+cT+1$ is irreducible over $\mathbb{F}_q$. Then, we can factor $H_3(r,s)$ over $\mathbb{F}_q$ as 
        \begin{align*}
        H_3(r,s) = &(rs^2 + (-2\alpha)rs + (2\alpha+b)s^2 + (-2\alpha^2-2b\alpha-c)r + (2\alpha^2+2b\alpha+2c)s + 1) 
        \\& \cdot (rs^2 + (-2\alpha')rs + (2\alpha'+b)s^2 + (-2\alpha'^2-2b\alpha'-c)r + (2\alpha'^2+2b\alpha'+ 2c)s + 1)  
        \\ & \cdot (rs^2 + (-2\alpha'')rs + (2\alpha''+b)s^2 + (-2\alpha''^2-2b\alpha''-c)r + (2\alpha''^2+2b\alpha''+2c)s + 1),
    \end{align*}
    where $\alpha, \alpha', \alpha''\in\mathbb{F}_{q^3}$ are roots to $T^3+bT^2+cT+1$.

    The $\mathbb{F}_q$-rational points of $H_3(r,s) = 0$ are in the intersection of these $3$ conjugate cubic curves defined over $\mathbb{F}_{q^3}$, which has at most $9$ points. Hence, the density of conic pairs satisfying the Poncelet tetragon condition is $0+O(q^{-2})$ in this pencil. We will prove in the next section that this density is exactly $0$.
\end{itemize}
\end{proof}

\section{Poncelet $n$-gons and $n$-torsion points of elliptic curves over function fields}\label{Sec:ngon}
\subsection{Connection between Poncelet $n$-gon and $n$-torsion points of a family of elliptic curves}
By the theorem of Griffiths and Harris (Theorem \ref{Thm:GH}), the conic pair \((\mathcal{A}, \mathcal{B})\) satisfies the Poncelet \(n\)-gon condition if and only if the point \((0, \sqrt{\det B})\) is an \(n\)-torsion point on the elliptic curve \(y^2 = \det(xA + B)\). To count the number of smooth conic pairs within a pencil generated by \(\mathcal{A}\) and \(\mathcal{B}\) that satisfy the Poncelet \(n\)-gon condition, we fix the first conic $\mathcal{A}$ and allow the second conic to vary across all members of the pencil. This process is equivalent to counting the number of \(n\)-torsion points on the elliptic curve \(y^2 = \det(xA + B)\) with \(x\)-coordinates in \(\mathbb{F}_q\). Furthermore, we show that varying the first conic in the pencil corresponds to selecting a different elliptic curve from a family of elliptic curves.  

The following results formalize these observations:

\begin{lem}\label{lem:3.1}
    Let \((\mathcal{A}, \mathcal{B})\) be a pair of smooth conics with transversal intersection, and let \(E: y^2 = \det(xA + B)\) be the corresponding elliptic curve. Define  
    \[
    n_A = \left|\{c \in \mathbb{F}_q : (\mathcal{A}, c \mathcal{A} + \mathcal{B}) \text{ satisfies the Poncelet \(n\)-gon condition, and } c \mathcal{A} + \mathcal{B} \text{ is smooth}\}\right|.
    \] 
    Then, 
    \[
    n_A = \frac{1}{2} \left|\{P \in E[n] \setminus E[2] : x(P) \in \mathbb{F}_q\}\right|.
    \]
\end{lem}

\begin{thm}\label{thm:3.2}
   Let \(\mathcal{P}\) be a pencil with intersection type \((1,1,1,1)\) or \((2,2)\), and \(E_\lambda\) be the Legendre family of elliptic curves:
    \[
    E_\lambda : y^2 = x(x - 1)(x - \lambda).
    \]
    We view $E_\lambda$ as an elliptic curve over the function field $\mathbb{F}_q(\lambda)$, and let \(E_{\lambda_0}\) denote the reduction of \(E_\lambda\) at a degree 1 place \(\lambda_0\in\mathbb{P}^1(\mathbb{F}_q)\). 

    Then, the number of smooth conic pairs in the pencil \(\mathcal{P}\) that satisfy the Poncelet \(n\)-gon condition, denoted by \(n^\mathcal{P}\), is given by
    \[
    n^\mathcal{P} = \frac{1}{2}\sum_{\lambda_0 \in \mathbb{F}_q \setminus \{0,1\}} \left|\{P \in E_{\lambda_0}[n] \setminus E_{\lambda_0}[2] : x(P) \in \mathbb{F}_q\}\right|.
    \]
\end{thm}

We also observe that \(E_\lambda\) has \(3\) degree $1$ places of bad reductions, specifically at \(0\), \(1\), and \(\infty\). This corresponds to the fact that there are \(3\) singular conics defined over \(\mathbb{F}_q\) in these pencils.

Similar results can be derived for pencils with other intersection types. For pencils of intersection type \((2,1,1)\) and \((4)\), the corresponding family of elliptic curves is 
\[
E_\lambda: y^2 = (x - \lambda)(x^2 - b),
\]
where \(b \in \mathbb{F}_q\) is a fixed constant that is not a square in \(\mathbb{F}_q\). 

For pencils with intersection type \((3,1)\), the corresponding family of elliptic curves is 
\[
E_\lambda: y^2 = \left(x - \frac{\alpha}{\alpha + \lambda}\right)\left(x - \frac{\alpha'}{\alpha' + \lambda}\right)\left(x - \frac{\alpha''}{\alpha'' + \lambda}\right),
\]
where \(\alpha \in \mathbb{F}_{q^3} \setminus \mathbb{F}_q\), and \(\alpha'\), \(\alpha''\) are the conjugates of \(\alpha\).

\begin{rmk}
The theorem above also shows that the number of smooth conic pairs that satisfy the Poncelet \(n\)-gon condition is the same in the pencils with intersection types \((1,1,1,1)\) and \((2,2)\), and also the same in the pencils with intersection type \((2,1,1)\) and \((4)\). This matches up with our computation for the Poncelet triangle and tetragon cases in Section \ref{triangle} and Section \ref{tetragon}. 
\end{rmk}

Now, we use the Griffiths-Harris Theorem (Theorem \ref{Thm:GH}) along with a change of coordinates to prove the above results. 

\begin{proof}[Proof of Lemma \ref{lem:3.1}]
    When counting the number of smooth conic pairs $(\mathcal{A}, c\mathcal{A}+\mathcal{B})$ that satisfy the Poncelet $n$-gon condition, we require that $c\neq \infty$, as otherwise \(c \mathcal{A} + \mathcal{B} = \mathcal{A}\) would not intersect transversally with \(\mathcal{A}\).

    Let \(E^c\) be the elliptic curve \(y^2 = \det(xA + (c A + B)) = \det((x + c)A + B)\). By the Griffiths-Harris theorem, the pair \((\mathcal{A}, c \mathcal{A} + \mathcal{B})\) satisfies the Poncelet \(n\)-gon condition if and only if the point \(P \in E^c\) with \(x(P) = 0\) is an \(n\)-torsion point of \(E^c\). Using the change of coordinates \(x' = x + c\), this is equivalent to requiring that the point \(P^c \in E\) with \(x(P^c) = c \in \mathbb{F}_q\) be an \(n\)-torsion point of the elliptic curve \(E: y^2 = \det(xA + B)\).

    Next, we exclude those \(c\) such that \(c \mathcal{A} + \mathcal{B}\) is a singular conic. If \(c \mathcal{A} + \mathcal{B}\) is singular, then \(y^2 = \det(c A + B) = 0\), which implies that the point \(P' = (c, 0)\) is a 2-torsion point of \(E\). The converse also holds. Therefore, we exclude all 2-torsion points from the counting.

    Finally, we divide the result by 2, since for any point \(P\), $-P\neq P$ also has the same $x$-coordinate.

\end{proof}


We now prove a criterion for two conics to intersect transversally before proceeding to the proof of Theorem \ref{thm:3.2}. An important observation is that if the field \(k\) has characteristic greater than \(2\), every conic is projectively equivalent to the conic \(\mathcal{I}: x^2 + y^2 + z^2 = 0\). Without loss of generality, we may assume that the first conic is \(\mathcal{I}\). This leads to the following result:

\begin{lem}
    The conic \(\mathcal{I}\) intersects \(\mathcal{B}\) transversally if and only if matrix representation of $\mathcal{B}$ has three distinct eigenvalues over \(\overline{k}\).
\end{lem}
\begin{proof}
    Let \(I\) and \(B\) denote the matrix representations of the conics \(\mathcal{I}\) and \(\mathcal{B}\) respectively.  
    
    Suppose \(\mathcal{I}\) and \(\mathcal{B}\) intersect transversally. Then the equation \(y^2 = \det(xI + B)\) defines an elliptic curve. Rewriting this equation, we have \(y^2 = (x - \eta_1)(x - \eta_2)(x - \eta_3)\) for some \(\eta_1, \eta_2, \eta_3 \in \overline{k}\). Since the curve is an elliptic curve, the roots \(\eta_1, \eta_2, \eta_3\) are distinct. Consequently, the eigenvalues of \(B\), which are \(-\eta_1, -\eta_2, -\eta_3\), are also distinct.

    Conversely, assume that \(B\) has three distinct eigenvalues \(-\eta_1, -\eta_2, -\eta_3 \in \overline{k}\). Then the matrix \(B + \eta_1 I\) has rank \(2\), as each eigenspace of \(B\) has dimension \(1\). This implies that the conic represented by \(B + \eta_1I\) is a union of two distinct lines. These two lines intersect \(\mathcal{I}\) at four distinct points. Since \(\mathcal{B}\) lies in the pencil generated by \(\mathcal{I}\) and \(B + \eta_1I\), it follows that \(\mathcal{B}\) intersects \(\mathcal{I}\) transversally at those four points.
\end{proof}

\begin{proof}[Proof of Theorem \ref{thm:3.2}]
    %
    By a projective change of coordinates, we may assume that the pencil \(\mathcal{P}\) is generated by the smooth conics \(\mathcal{I}\) and \(\mathcal{B}\), where the corresponding matrix \(B\) has three distinct eigenvalues \(-\eta_1, -\eta_2, -\eta_3\). Since \(\mathcal{P}\) has intersection type \((1,1,1,1)\) or \((2,2)\), the three singular conics in this pencil are defined over \(\mathbb{F}_q\). Thus, we have \(\eta_1, \eta_2, \eta_3 \in \mathbb{F}_q\). By replacing the generators of \(\mathcal{P}\) with \(\mathcal{I}\) and \(\mathcal{B} + \eta_1 \mathcal{I}\), we may assume \(\eta_1 = 0\). Additionally, by scaling the matrix \(B\) (choosing a different representation of the conic \(\mathcal{B}\)), we can set \(\eta_2 = 1\). 

    The singular conics in the pencil \(\mathcal{P}\) are then given by \(\mathcal{B}, \mathcal{B} + \mathcal{I}, \mathcal{B} + \eta_3 \mathcal{I}\). We fix the first smooth conic to be \(\mathcal{B} + \lambda \mathcal{I}\) for \(\lambda \in \mathbb{P}^1(\mathbb{F}_q) \setminus \{0, 1, \eta_3\}\) and consider the corresponding elliptic curve 
    \[
    E^\lambda: y^2 = \det(x(B + \lambda I) + B),
    \]
    with \(E^\infty: y^2 = \det(xI + B)\). By Lemma \ref{lem:3.1}, the number of smooth conic pairs in $\mathcal{P}$ with the first conic being \(\mathcal{B} + \lambda \mathcal{I}\) that satisfy the Poncelet \(n\)-gon condition is 
    \[
    n_{\mathcal{B} + \lambda \mathcal{I}} = \frac{1}{2} \left| \{ P \in E^\lambda[n] \setminus E^\lambda[2] : x(P) \in \mathbb{F}_q \} \right|.
    \]
    The total number of smooth conic pairs in \(\mathcal{P}\) satisfying the Poncelet \(n\)-gon condition is 
    \[
    n^\mathcal{P} = \sum_{\lambda \in \mathbb{P}^1(\mathbb{F}_q) \setminus \{0, 1, \eta_3\}} n_{\mathcal{B} + \lambda \mathcal{I}}.
    \]
    
    Next, we show that the family of elliptic curves \(\{E^\lambda : \lambda \in \mathbb{P}^1(\mathbb{F}_q) \setminus \{0, 1, \eta_3\}\}\) can be replaced by the Legendre family \(\{E_{\lambda_0} : \lambda_0 \in \mathbb{P}^1(\mathbb{F}_q) \setminus \{0, 1, \infty\}\}\), completing the proof.

    Consider the elliptic curve \(E^\lambda: y^2 = \det(x(B + \lambda I) + B)\) for \(\lambda \in \mathbb{P}^1(\mathbb{F}_q) \setminus \{0, 1, \eta_3\}\). Since the conic \(\mathcal{B} + \lambda \mathcal{I}\) is non-singular, \(\det(B + \lambda I) \neq 0\), and we can rewrite the defining equation of \(E^\lambda\) as 
    \[
    E^\lambda : y^2 = \det(B + \lambda I) \det(xI + B(B + \lambda I)^{-1}).
    \]
    Performing a coordinate change \(y' = y / \sqrt{\det(B + \lambda I)}\), which does not affect the field of definition of the \(x\)-coordinates, we may rewrite the equation of $E^\lambda$ as 
    \[
    E^\lambda : y^2 = x \left(x - \frac{1}{-1 + \lambda}\right) \left(x - \frac{\eta_3}{-\eta_3 + \lambda}\right).
    \]
    Using another coordinate change \(x' = (-1 + \lambda)x\), \(y' = y / \sqrt{(-1 + \lambda)^3}\), which again preserves the field of definition of the \(x\)-coordinate, we may replace the equation for $E^\lambda$ as
    \[
    E^\lambda : y^2 = x(x - 1) \left(x - \frac{-\eta_3 + \eta_3\lambda}{-\eta_3 + \lambda}\right).
    \]

    The map \(\lambda \mapsto \frac{-\eta_3 + \eta_3\lambda}{-\eta_3 + \lambda}\) is a bijection between \(\mathbb{P}^1(\mathbb{F}_q) \setminus \{0, 1, \eta_3\}\) and \(\mathbb{P}^1(\mathbb{F}_q) \setminus \{0, 1, \infty\}\). Hence, the family \(\{E^\lambda : \lambda \in \mathbb{P}^1(\mathbb{F}_q) \setminus \{0, 1, \eta_3\}\}\) is precisely the Legendre family \(\{E_{\lambda_0} : \lambda_0 \in \mathbb{P}^1(\mathbb{F}_q) \setminus \{0, 1, \infty\}\}\).

\end{proof}

\subsection{Proof of Theorem \ref{thm:n-gon}}
We now prove Theorem \ref{thm:n-gon}, assuming that Theorem \ref{thm:main1} holds for all odd integers \( n \) coprime to \( q \) for the elliptic curves $E_\lambda$ in Theorem \ref{thm:3.2}. For the reader's convenience, we restate Theorem \ref{thm:n-gon} below:

\begin{thm*}(Density of smooth conic pairs satisfying the Poncelet \( n \)-gon condition) \par
    Let \(\mathbb{F}_q\) be a finite field with characteristic greater than \( 3 \), and let \( n \) be an odd integer coprime to \( q \). Then,
    \[
    \frac{|\Gamma_n|}{|\Psi|} = \frac{d(n)-1}{q} + O(q^{-3/2}),
    \]
    where \( d(n) \) is the number of divisors of \( n \), \( \Psi \) is the set of ordered pairs of smooth conics with transversal intersection, and \( \Gamma_n \subset \Psi\) is the subset of pairs of smooth conics satisfying the Poncelet \( n \)-gon condition.
\end{thm*}

\begin{proof}
    Let \( E_\lambda \) denote one of the elliptic curves defined in Theorem \ref{thm:3.2}. The number of smooth conic pairs in a pencil \( \mathcal{P} \) that satisfy the Poncelet \( n \)-gon condition is given by
    \[
    n^\mathcal{P} = \frac{1}{2} \sum_{\lambda_0} \left| \{ P \in E_{\lambda_0}[n] \setminus E_{\lambda_0}[2] : x(P) \in \mathbb{F}_q \} \right|,
    \]
    where \( \lambda_0 \) ranges over elements of \(\mathbb{P}^1(\mathbb{F}_q)\) such that \( E_\lambda \) has good reduction at \( \lambda_0 \).

    Since \( n \) is odd, the only \( n \)-torsion point on \( E_{\lambda_0} \) that is also a \( 2 \)-torsion point is the identity \( O \). Thus, we can simplify
    \[
    n^\mathcal{P} = \frac{1}{2} \sum_{\lambda_0} \left| \{ P \in E_{\lambda_0}[n] \setminus \{O\} : x(P) \in \mathbb{F}_q \} \right| = \sum_{\lambda_0} r(n, \lambda_0),
    \]
    where \( r(n, \lambda_0) \) is the number of \(\mathbb{F}_q\)-roots of the \( n \)-th division polynomial \( \Lambda(x, \lambda_0) \) of \( E_{\lambda_0} \).

    By Theorem \ref{thm:main1}, we have
    \[
    \sum_{\lambda_0} r(n, \lambda_0) = (d(n)-1)q + O(\sqrt{q}),
    \]
    where \( d(n) \) is the number of divisors of \( n \). Therefore, the density of smooth conic pairs in \( \mathcal{P} \) satisfying the Poncelet \( n \)-gon condition is $\frac{d(n)-1}{q} + O(q^{-3/2})$.

    Since this result holds for all pencils \( \mathcal{P} \), we conclude that the density of smooth conic pairs satisfying the Poncelet \( n \)-gon condition is also
    \[
    \frac{|\Gamma_n|}{|\Psi|} = \frac{d(n)-1}{q} + O(q^{-3/2}).
    \]
\end{proof}

\section{The $n$-torsion points of a family of elliptic curves}
In this section, we develop the tools necessary to prove Theorem \ref{thm:main1}. Specifically, we analyze the arithmetic of the \( n \)-torsion points of an elliptic curve \( E \) defined over the function field \( K = \mathbb{F}_q(\lambda) \). For an elliptic curve \( E_{\lambda_0} \) in the family, which arises as the reduction of \( E \) at a place \( \lambda_0 \), we control the behavior of the \( n \)-torsion points on \( E_{\lambda_0} \) via the Frobenius at \( \lambda_0 \) in the Galois extension \( K(E[n])/K \), where \( K(E[n]) \) is the \( n \)-torsion point field defined earlier in this paper.

\subsection{The $n$-torsion point field}
Let \( E \) be an elliptic curve defined over a field \( K \), and let \( x, y \) be the regular functions appearing in the Weierstrass equation of \( E \). As noted previously, the \( n \)-torsion point field \( K(E[n]) \) is the field extension of \( K \) obtained by adjoining all the \( x \)- and \( y \)-coordinates of points in \( E[n] \).

We assume that \( n \) is coprime to \( \operatorname{char}(K) \). Then, the \( n \)-torsion group \( E[n] \) is isomorphic to \( \mathbb{Z}/n\mathbb{Z} \times \mathbb{Z}/n\mathbb{Z} \). The action of \( \operatorname{Gal}(K^{\text{sep}}/K) \) on \( E[n] \) induces the representation
\[
\rho_{E, n}: \operatorname{Gal}(K^{\text{sep}}/K) \to \operatorname{Aut}(E[n]) \simeq \operatorname{GL}_2(\mathbb{Z}/n\mathbb{Z}),
\]
where the kernel of \( \rho_{E, n} \) is \( \operatorname{Gal}(K^{\text{sep}}/K(E[n])) \). Consequently, the extension \( K(E[n])/K \) is Galois, and we can identify the Galois group with a subgroup of \( \operatorname{GL}_2(\mathbb{Z}/n\mathbb{Z}) \) via the induced representation:
\[
\rho_{E, n}: \Gal(K(E[n])/K) \hookrightarrow \operatorname{GL}_2(\mathbb{Z}/n\mathbb{Z}).
\]

Next, let \( l \) be a prime coprime to \( q \). The Tate module \( T_l(E) = \varprojlim E[l^k] \simeq \mathbb{Z}_l \times \mathbb{Z}_l \) encodes the action of the Galois group on the \( l \)-power torsion points. We define \( K(E[l^\infty]) \) as the field extension obtained by adjoining the coordinates of all points in \( \bigcup_k E[l^k] \), equivalently expressed as \( K(E[l^\infty]) = \bigcup_k K(E[l^k]) \). The Galois group \( \operatorname{Gal}(K(E[l^\infty])/K) \) acts on \( T_l(E) \) and induces the representation
\[
\hat{\rho}_{E, l}: \operatorname{Gal}(K(E[l^\infty])/K) \hookrightarrow \operatorname{GL}_2(\mathbb{Z}_l).
\]
This action is compatible with the multiplication-by-\( l \) map, yielding the congruence \( \rho_{E, l^k} \equiv \hat{\rho}_{E, l} \pmod {l^k} \).

From this point forward, we view an element of \( \operatorname{Gal}(K(E[n])/K) \) as a matrix in \( \operatorname{GL}_2(\mathbb{Z}/n\mathbb{Z}) \) via the representation \( \rho_{E, n} \). We will use this matrix to understand the arithmetic properties of $E[n]$.

\begin{lem}
    Let \( \mu_n \) denote the \( n \)-th roots of unity in \( K^{\text{sep}} \). Then, \( K(\mu_n) \subset K(E[n]) \), and for any \( \sigma \in \Gal(K(E[n])/K) \), we have \( \zeta^\sigma = \zeta^{\det \sigma} \) for all $\zeta\in \mu_n$.
\end{lem}
\begin{proof}
We consider Weil's pairing
\[
e_n: E[n] \times E[n] \to \mu_n,
\]
which is compatible with \( \Gal(K^{\text{sep}}/K) \), meaning \( e_n(P^\sigma, Q^\sigma) = e_n(P, Q)^\sigma \). Since \( \sigma \in \Gal(K^{\text{sep}}/K(E[n])) \) fixes all \( n \)-torsion points, it also fixes every element in \( \mu_n \). Thus, \( K(\mu_n) \subset K(E[n]) \).

Now, let \( \{P, Q\} \) be a basis for \( E[n] \), and let \( \sigma \in \Gal(K(E[n])/K) \subset \operatorname{GL}_2(\mathbb{Z}/n\mathbb{Z}) \) be represented by the matrix
\[
\sigma = \begin{bmatrix} a & b \\ c & d \end{bmatrix}.
\]
We know that $e_n(P,Q) = \zeta_n$ is a primitive $n$-th root of unity, and we compute
\[
e_n(P, Q)^\sigma = e_n(P^\sigma, Q^\sigma) = e_n(aP + cQ, bP + dQ) = e_n(P, Q)^{ad - bc} = e_n(P, Q)^{\det \sigma}.
\]
Thus, \( \sigma \in \Gal(K(E[n])/K) \) acts on $\zeta_n$ via $\zeta_n^\sigma = \zeta_n^{\det\sigma}$, and we have $\zeta^\sigma = \zeta^{\det\sigma}$ for all $\zeta\in \mu_n$. 
\end{proof}

Now, assume \( K = \mathbb{F}_q(C) \) is the function field of a smooth curve \( C \) defined over \( \mathbb{F}_q \). Let \( v \) be a degree 1 place of \( K \), and let \( E_v \) be the reduction of \( E \) at \( v \), which is a cubic curve defined over the residue field \( \kappa_v = \mathbb{F}_q \). If \( E \) has good reduction at \( v \), the Néron-Ogg-Shafarevich criterion implies that \( v \) is unramified in the extension \( K(E[n])/K \). Let \( \text{Frob}_v \in \Gal(K(E[n])/K) \) denote the Frobenius of \( v \), viewed as an element of \( \operatorname{GL}_2(\mathbb{Z}/n\mathbb{Z}) \), and let \( \sigma_{q,v} \) be the Frobenius-\( q \)-map of the extension \( \Gal(\kappa_v(E_v[n])/\kappa_v) \subset \operatorname{GL}_2(\mathbb{Z}/n\mathbb{Z}) \). By construction, these elements act identically on \( E_v[n] \), so we have \( \text{Frob}_v = \sigma_{q,v} \) as matrices in \( \operatorname{GL}_2(\mathbb{Z}/n\mathbb{Z}) \).

To compute \( r(n, v) \), the number of \( \mathbb{F}_q \)-roots of the \( n \)-th division polynomial \( \Lambda_n(x, v) \) of the elliptic curve \( E_v \), it suffices to count the points \( P \in E_v[n] \) such that \( \sigma_{q,v}(P) = \pm P \). Thus, the arithmetic of \( E_v[n] \) is controlled by the conjugacy class of \( \sigma_{q,v} \). We will show that only finitely many conjugacy classes of the matrix \( \sigma_{q,v} \) contribute to non-zero \( r(n, v) \). For each such conjugacy class, we will use the Chebotarev density theorem to count the number of degree 1 places \( v \) such that \( \text{Frob}_v = \sigma_{q,v} \) is in that class. 

\subsection{The Galois group of the $n$-torsion point field}
In this section, we investigate the size of the Galois group \( \Gal(K(E[n])/K) \) as a subgroup of \( \text{GL}_2(\mathbb{Z}/n\mathbb{Z}) \). Our main goal is to prove that if \( K = \mathbb{F}_q(\lambda) \) and $E: y^2 = x(x-1)(x-\lambda)$ is the Legendre curve, then the geometric Galois group of \( K(E[n])/K \) is isomorphic to \( \operatorname{SL}_2(\mathbb{Z}/n\mathbb{Z}) \) for all odd integers \( n \) coprime to \( q \).

\subsubsection{The Galois group of $K(E[l])/K$}

We begin by recalling some classical results in the number field setting and then explain their function field analogues.

If \( K \) is a number field and \( E \) is an elliptic curve over \( K \) without complex multiplication, a theorem of Serre shows that the natural representation
\[
\rho_{E, l}: \Gal(K(E[l])/K) \hookrightarrow \operatorname{GL}_2(\mathbb{Z}/l\mathbb{Z})
\]
is surjective for all but finitely many primes \( l \). (See \cite[Section 4.2, Theorem 2]{Serre}.) The primes for which surjectivity fails are called \textbf{exceptional primes}. Let \( c(E, K) \) denote the smallest integer such that for all primes \( l \ge c(E, K) \), \( l \) is not exceptional. Mazur showed that for a semi-stable elliptic curve over \( \mathbb{Q} \), we have \( c(E, K) \le 11 \) \cite{Mazur}, and Duke proved that almost all elliptic curves over \( \mathbb{Q} \) have no exceptional primes \cite{Duke}.

In the function field case the situation is different. In particular, the sub-extension \( K(\mu_n)/K \) may not have degree \( \varphi(n) \), so we expect the full Galois group \( \Gal(K(E[n])/K) \) to be smaller. To be precise, let \( K = \mathbb{F}_q(C) \) be the function field of a smooth curve \( C \) over \( \mathbb{F}_q \). We say that an elliptic curve \( E \) over \( K \) is \textbf{non-isotrivial} if its \( j \)-invariant is transcendental over \( \mathbb{F}_q \). In this case, we have \( \text{End}_K(E) = \text{End}_{\overline{K}}(E) \simeq \mathbb{Z} \), which is the function field analogue of having no complex multiplication.

We call $F = K(E[n])\cap \overline{\mathbb{F}_q}$ as the \textbf{field of constants} of $K(E[n])$, and the Galois group of $K(E[n])$ over $F(\mathcal{C})$ as the \textbf{geometric Galois group} of $K(E[n])$ over $K$. We will show later that the field of constants of $K(E[n])$ is $\mathbb{F}_q(\mu_n)$. 

There is a short exact sequence
\[
1 \to \Gal(K(E[n])/K(\mu_n)) \to \Gal(K(E[n])/K) \xrightarrow{\pi} \Gal(K(\mu_n)/K) \to 1.
\]
When \( n \) is coprime to \( q \), we have 
\[
\Gal(K(\mu_n)/K) \simeq \langle q \rangle \subset (\mathbb{Z}/n\mathbb{Z})^\times.
\]
By viewing \( \Gal(K(E[n])/K) \) as a subgroup of \( \operatorname{GL}_2(\mathbb{Z}/n\mathbb{Z}) \) and identifying \( \pi \) with the determinant map, it follows that \( \Gal(K(E[n])/K(\mu_n)) \) is a subgroup of \( \operatorname{SL}_2(\mathbb{Z}/n\mathbb{Z}) \).

We now state Igusa's theorem, which provides the function field analogue of Serre's result. We quote the following statement from \cite[Theorem 3.12]{Igusa}, which is an intermediate step toward the full statement of Igusa's theorem.

\begin{thm}(Igusa) \par
    Let \( K \) be a function field and let \( E \) be a non-isotrivial elliptic curve over \( K \). Then, for all but finitely many primes \( l \) coprime to \( q \), we have 
    \[
    \Gal(K(E[l])/K) = \Gamma_l,
    \]
    where \( \Gamma_l \subset \operatorname{GL}_2(\mathbb{Z}/l\mathbb{Z}) \) is defined by the short exact sequence
    \[
    1 \to \operatorname{SL}_2(\mathbb{Z}/l\mathbb{Z}) \to \Gamma_l \xrightarrow{\det} \langle q \rangle \to 1.
    \]
    More explicitly, 
    \[
    \Gamma_l = \{ \sigma \in \operatorname{GL}_2(\mathbb{Z}/l\mathbb{Z}) : \det \sigma \equiv q^i \pmod{l} \text{ for some } i \}.
    \]
\end{thm}

A theorem of Cojocaru and Hall \cite{Hall} provides an upper bound for the exceptional primes \( c(E, K) \) in the function field setting. In particular, we obtain
\[
c(E, K) \leq 2 + \max\left\{ l \text{ prime: } \frac{1}{12}\Bigl[l - (6 + 3e_2 + 4e_3)\Bigr] \leq \text{genus}(K) \right\},
\]
where \( e_2 = 1 \) if \( l \equiv 1 \pmod{4} \) and \( e_2 = -1 \) otherwise, and \( e_3 = 1 \) if \( l \equiv 1 \pmod{3} \) and \( e_3 = -1 \) otherwise.

To address the Poncelet \( n \)-gon problem, Theorem \ref{thm:3.2} suggests that we focus on the Legendre curve
\[
E: y^2 = x(x-1)(x-\lambda)
\]
defined over the function field \( K = \mathbb{F}_q(\lambda) \). We begin by quoting a result of Yu \cite[Theorem 5.1]{Hall2}, which shows that there are no exceptional primes in this case. Using this result, we then deduce that the geometric Galois group of \( K(E[n])/K \) is isomorphic to \( \operatorname{SL}_2(\mathbb{Z}/n\mathbb{Z}) \), with the field of constants of \( K(E[n]) \) being \( \mathbb{F}_q(\mu_n) \).

\begin{thm} (Yu) \label{thm:Yu}\par
    Let $l$ be an odd prime, and $E: y^2 = x(x-1)(x-\lambda)$ be the Legendre curve defined over $K = \mathbb{F}_q(\lambda)$. If $l$ is coprime to $\text{char } K$, then $K(E[l])/K(\mu_l)$ is a geometric extension with Galois group $\operatorname{SL}_2(\mathbb{Z}/l\mathbb{Z})$. As a corollary, we have $\Gal(K(E[l)/K) = \Gamma_l$. 
\end{thm}

\subsubsection{The Galois group of \( K(E[l^\infty])/K \)}
Next, we show that the Galois group of \( K(E[l^\infty])/K(\mu_{l^\infty}) \) is \( \operatorname{SL}_2(\mathbb{Z}_l) \). This, in turn, implies that 
\[
\Gal\bigl(K(E[l^k])/K(\mu_{l^k})\bigr) \simeq \operatorname{SL}_2(\mathbb{Z}/l^k\mathbb{Z})
\]
for all \( k \ge 1 \). We begin with the following algebraic lemma (see \cite[Chapter IV, Section 3.4, Lemma 3]{Serre 2}):

\begin{lem}
    Let \( l \ge 5 \) be a prime number, and let \( H \) be a closed subgroup of \( \operatorname{SL}_2(\mathbb{Z}_l) \) that surjects onto \( \operatorname{SL}_2(\mathbb{Z}/l\mathbb{Z}) \). Then, \( H = \operatorname{SL}_2(\mathbb{Z}_l) \).
\end{lem}

\begin{cor}\label{cor: l^k_galois_group}
    Let \( K = \mathbb{F}_q(\lambda) \) and $E: y^2 = x(x-1)(x-\lambda)$
    be the Legendre curve over \( K \). If \( l \) is an odd prime coprime to \( q \), then the Galois group of \( K(E[l^\infty])/K(\mu_{l^\infty}) \) is $\operatorname{SL}_2(\mathbb{Z}_l)$.
\end{cor}

\begin{proof}
    There is a natural surjection
    \[
    \Gal\bigl(K(E[l^\infty])/K(\mu_{l^\infty})\bigr) \to \Gal\bigl(K(E[l])/K(\mu_l)\bigr)
    \]
    induced by the reduction map \( \operatorname{SL}_2(\mathbb{Z}_l) \to \operatorname{SL}_2(\mathbb{Z}/l\mathbb{Z}) \). By Theorem \ref{thm:Yu}, we have $\Gal\bigl(K(E[l])/K(\mu_l)\bigr) \simeq \operatorname{SL}_2(\mathbb{Z}/l\mathbb{Z})$. For \( l \ge 5 \), it then follows from the above lemma that 
    \[
    \Gal\bigl(K(E[l^\infty])/K(\mu_{l^\infty})\bigr) = \operatorname{SL}_2(\mathbb{Z}_l).
    \]

    For \( l = 3 \), we proceed as follows. First, we show that
    \[
    \begin{bmatrix} 1 & 1 \\ 0 & 1 \end{bmatrix} \in \Gal\bigl(K(E[3^\infty])/K(\mu_{3^\infty})\bigr).
    \]
    This inclusion follows from Tate's uniformization of \( E \) at a place of split multiplicative reduction (see \cite[Corollary 3.7]{Igusa} for details). Next, we will show that $
    \begin{bmatrix} 0 & 1 \\ -1 & 0 \end{bmatrix} \in \Gal\bigl(K(E[3^\infty])/K(\mu_{3^\infty})\bigr)$, and this implies that 
    \[
    \begin{bmatrix} 1 & 0 \\ 1 & 1 \end{bmatrix} = \begin{bmatrix} 0 & 1 \\ -1 & 0 \end{bmatrix}
    \begin{bmatrix} 1 & 1 \\ 0 & 1 \end{bmatrix}^2
    \begin{bmatrix} 0 & -1 \\ 1 & 0 \end{bmatrix}
    \in \Gal\bigl(K(E[3^\infty])/K(\mu_{3^\infty})\bigr).\]
    Since the matrices \( \begin{bmatrix} 1 & 1 \\ 0 & 1 \end{bmatrix} \) and \( \begin{bmatrix} 1 & 0 \\ 1 & 1 \end{bmatrix} \) generate \( \operatorname{SL}_2(\mathbb{Z}_l) \) (see \cite[Chapter XIII, Lemma 8.1]{Lang}), it follows that $
    \Gal\bigl(K(E[3^\infty])/K(\mu_{3^\infty})\bigr) = \operatorname{SL}_2(\mathbb{Z}_3)$. 

    Now, we show that $
    \begin{bmatrix} 0 & 1 \\ -1 & 0 \end{bmatrix} \in \Gal\bigl(K(E[3^\infty])/K(\mu_{3^\infty})\bigr)$. It suffices to show that for every $n$, there is $\sigma\in \Gal\bigl(K(E[3^\infty])/K(\mu_{3^\infty})\bigr)$ such that $\sigma\equiv \begin{bmatrix}
        0 & 1 \\ -1 & 0 
    \end{bmatrix}\pmod{3^n}$. Since $\Gal\bigl(K(E[3])/K(\mu_3)\bigr) = \text{SL}_2(\mathbb{Z}/3\mathbb{Z})$, we can find $\sigma\in \Gal\bigl(K(E[3^\infty])/K(\mu_{3^\infty})\bigr)$ whose reduction modulo $3$ is $\begin{bmatrix} 0 & 1 \\ -1 & 0\end{bmatrix}$. Suppose that for some \( n \ge 1 \), there exists \( \sigma \in \Gal(K(E[3^\infty])/K(\mu_{3^\infty})) \) such that $\sigma \equiv \begin{bmatrix} 0 & -1 \\ 1 & 0 \end{bmatrix} \pmod{3^n}$.
    We write $\sigma = \begin{bmatrix} 0 & -1 \\ 1 & 0 \end{bmatrix} + 3^n y$,
    for some \( y \in \operatorname{M}_2(\mathbb{Z}_3) \). Then, we compute that 
    $$\sigma^3 = \begin{bmatrix}
        0 & 1 \\ -1 &0
    \end{bmatrix} + 3^{n+1}y^3 \equiv \begin{bmatrix}
        0 & 1 \\ -1 &0
\end{bmatrix} \pmod{3^{n+1}}.$$
    Thus, we have  \( \sigma^3 \in\Gal(K(E[3^\infty])/K(\mu_{3^\infty})) \), whose reduction modulo \( 3^{n+1} \) is \( \begin{bmatrix} 0 & 1 \\ -1 & 0 \end{bmatrix} \). This completes the induction and hence the proof for the \( l = 3 \) case.

    
\end{proof}

\subsubsection{The Galois group of \( K(E[n])/K \)}

\begin{thm}\label{thm: Galois_n_torsion}
    Let \( K = \mathbb{F}_q(\lambda) \) and $E: y^2 = x(x-1)(x-\lambda)$
    be the Legendre curve over \( K \). If \( n \) is an odd integer coprime to \( q \), then
    \[
    \Gal(K(E[n])/K(\mu_n)) = \operatorname{SL}_2(\mathbb{Z}/n\mathbb{Z}).
    \]
\end{thm}

\begin{proof}
    Let $n = \prod_{1}^k l_i^{a_i}$ be the prime factorization of $n$. Since \( K(\mu_n) = \mathbb{F}_{q^m}(\lambda) \) for some \( m \), Corollary \ref{cor: l^k_galois_group} applies with $K$ replaced by $K(\mu_n)$, and we have
    \[
    \Gal\bigl(K(\mu_n, E[l_i^{a_i}])/K(\mu_n)\bigr) \simeq \operatorname{SL}_2(\mathbb{Z}/l_i^{a_i}\mathbb{Z})
    \]
    for each prime factor \( l_i \). We now show that the fields \( K(\mu_n, E[l_i^{a_i}]) \) have trivial pairwise intersection over \( K(\mu_n) \). Therefore,
    \[
    \Gal(K(E[n])/K(\mu_n)) \simeq \prod_{i=1}^k \Gal\bigl(K(\mu_n, E[l_i^{a_i}])/K(\mu_n)\bigr)
    \simeq \prod_{i=1}^k \operatorname{SL}_2(\mathbb{Z}/l_i^{a_i}\mathbb{Z})
    \simeq \operatorname{SL}_2(\mathbb{Z}/n\mathbb{Z}).
    \]
    For further details, we refer to \cite[Figures 5.1 and 5.2]{Adel}, which list all normal subgroups of \( \operatorname{SL}_2(\mathbb{Z}/l^m\mathbb{Z}) \). In those figures, we see that the nontrivial quotient groups of \( \operatorname{SL}_2(\mathbb{Z}/l^m\mathbb{Z}) \) have distinct cardinalities for different primes \( l \). This forces the pairwise intersections \( K(\mu_n, E[l_i^{a_i}]) \cap K(\mu_n, E[l_j^{a_j}]) \) (for \( i \neq j \)) to be trivial over \( K(\mu_n) \).

\end{proof}

\begin{cor}
    With the same setting as above, the field of constants of \( K(E[n]) \) is \( \mathbb{F}_q(\mu_n) \). Thus, the geometric Galois group of \( K(E[n])/K \) is $
    \Gal(K(E[n])/K(\mu_n)) = \operatorname{SL}_2(\mathbb{Z}/n\mathbb{Z})$.
\end{cor}

\begin{proof}
    Let $F = \overline{\mathbb{F}_q} \cap K(E[n])$ be the field of constants of \( K(E[n]) \). By Weil's pairing, we have \( \mu_n \subset K(E[n])\) and thus \(\mu_n\subset F \).
    Let \( K' = F(\lambda) \). Since \( K(\mu_n) \subset K' \) and \( K'(E[n]) = K(E[n]) \), we have
    \[
    |\operatorname{SL}_2(\mathbb{Z}/n\mathbb{Z})| = [K(E[n]):K(\mu_n)] = [K'(E[n]):K(\mu_n)] = [K'(E[n]):K'(\mu_n)] \cdot [K'(\mu_n):K(\mu_n)].
    \]
    But \( \Gal(K'(E[n])/K'(\mu_n)) \simeq \Gal(K'(E[n])/K'(\mu_n)) = \operatorname{SL}_2(\mathbb{Z}/n\mathbb{Z}) \), so comparing cardinalities forces \( K'(\mu_n) = K(\mu_n) \). Consequently, we have $F = \mathbb{F}_q(\mu_n)$.
\end{proof}

Lastly, we show that if \( E \) is one of the other two elliptic curves considered in Theorem \ref{thm:3.2}, then for all odd integers \( n \) coprime to \( q \), we still have 
\[
\Gal(K(E[n])/K(\mu_n)) = \operatorname{SL}_2(\mathbb{Z}/n\mathbb{Z}).
\]

If \( E \) is one of these curves, then \( E \) is isomorphic to the Legendre curve over \( \mathbb{F}_{q^3}(\lambda) \) or \( \mathbb{F}_{q^2}(\lambda) \). Let $K' = \mathbb{F}_{q^3}(\lambda)$ or $\mathbb{F}_{q^2}(\lambda)$. We have $\Gal(K'(E[n])/K'(\mu_n)) = \operatorname{SL}_2(\mathbb{Z}/n\mathbb{Z})
$ by Theorem \ref{thm: Galois_n_torsion}. 

We consider the following diagram
$$
\begin{tikzcd}
 & {K'(E[n])}\\
{K(E[n])} \arrow[ru, no head]& \\
& K'(\mu_n) \arrow[uu, no head] \\
K(\mu_n) \arrow[ru, "\text{deg} \le 3"', no head] \arrow[uu, no head] & 
\end{tikzcd}
$$
Since \( [K'(\mu_n):K(\mu_n)] \le 3 \), the intersection \( K(E[n]) \cap K'(\mu_n) \) is either \( K(\mu_n) \) or \( K'(\mu_n) \). 

In the first case, we have
\[
\Gal(K(E[n])/K(\mu_n)) \simeq \Gal(K'(E[n])/K'(\mu_n)) = \operatorname{SL}_2(\mathbb{Z}/n\mathbb{Z}).
\]

In the second case, we have \( K'(\mu_n) \subset K(E[n]) \) so that \( K(E[n]) = K'(E[n]) \), and therefore
\[
\Gal(K(E[n])/K(\mu_n)) \supset \Gal(K(E[n])/K'(\mu_n)) =  \Gal(K'(E[n])/K'(\mu_n)) =\operatorname{SL}_2(\mathbb{Z}/n\mathbb{Z}).
\]
Since \( \Gal(K(E[n])/K(\mu_n)) \) is a subgroup of \( \operatorname{SL}_2(\mathbb{Z}/n\mathbb{Z}) \), we conclude that 
\[
\Gal(K(E[n])/K(\mu_n)) = \operatorname{SL}_2(\mathbb{Z}/n\mathbb{Z})
\]
in all cases.

\subsection{Main theorem on the $l^k$-torsion points of a family of elliptic curves}
We now prove Theorem \ref{thm:main1} using the Galois groups of the \(n\)-torsion point fields and Chebotarev Density Theorem. In this section, we focus on the \(l^k\)-torsion points, and generalize the results to arbitrary \(n\)-torsion points in the next section.

Let $K = \mathbb{F}_q(\lambda)$ be a function field with $\operatorname{char}(K)>3$, and let \(E\) be a non-isotrivial elliptic curve defined over \(K\).  For any degree 1 place \(\lambda_0 \in \mathbb{P}^1(\mathbb{F}_q)\), we
denote by \(E_{\lambda_0}\) the reduction of \(E\) at \(\lambda_0\). We also define 
$$r(n, \lambda_0) = \# \{\mathbb{F}_q\text{-roots to the n-th division polynomial }\Lambda_n(x, \lambda_0) \text{ of }E_{\lambda_0}\},$$
$$R^+(n, \lambda_0) = \{P\in E_{\lambda_0}[n]: \sigma_{q,\lambda_0}(P) = P\},$$
$$R^+(n, \lambda_0) = \{P\in E_{\lambda_0}[n]: \sigma_{q,\lambda_0}(P) = -P\},$$
$$r^+(n, \lambda_0) = \Big|\{P\in E_{\lambda_0}[n]\setminus \{O\}: \sigma_{q,\lambda_0}(P) = P\}/(-P\sim P)\Big|,$$
$$r^-(n, \lambda_0) = \Big|\{P\in E_{\lambda_0}[n]\setminus \{O\}: \sigma_{q,\lambda_0}(P) = -P\}/(-P\sim P)\Big|,$$
where we write $/(-P \sim P)$ to indicate that we only count each pair of a point and its negative once. It is clear that $r(n, \lambda_0) = r^+(n, \lambda_0)+r^-(n,\lambda_0)$. If $n$ is and odd integer, we also have $r^\pm(n, \lambda_0) = \frac{|R^{\pm}(n, \lambda_0) - 1|}{2}$.

Since we are only concerned with degree 1 primes, we now quote a version of the Chebotarev Density Theorem for such primes (see, e.g., \cite[Lemma 1]{Chebota}):

\begin{thm}(Chebotarev Density Theorem for Degree 1 Primes)\label{thm:Chebotarev}\par
    Let \( K \) be a function field over \( \mathbb{F}_q \), and \( L \) be a finite Galois extension of \( K \). Denote by \( L_0 = L \cap \overline{\mathbb{F}_q} \) the field of constants of \( L \), and let \( \mathfrak{C} \) be a conjugacy class in \( \Gal(L/K) \). If \( C_1 \) denotes the set of unramified degree 1 primes of \( K \) whose Frobenius lies in \( \mathfrak{C} \), then
    \[
    |C_1| = \frac{|\mathfrak{C}|}{[L : L_0 K]} \, q + O(\sqrt{q}).
    \]
\end{thm}

In our setup, we will use $K = \mathbb{F}_q(\lambda)$, $L = K(E[n])$,  $L_0 = \mathbb{F}_q(\mu_n)$, $L_0K = K(\mu_n)$, and $[L: L_0K]$ is the order of the geometric Galois group of $K(E[n])/K$.  

\begin{thm} (Main result on the $l^k$-torsion points of a family of elliptic curves) \label{thm: main l^k}\par
   Let $K = \mathbb{F}_q(\lambda)$ be a function field with $\operatorname{char}(K)>3$, and \(E\) be a non-isotrivial elliptic curve defined over \(K\). Let $l$ be an odd prime number coprime to $q$. If the geometric Galois group of $K(E[l^k])/K$ is  $\operatorname{SL}_2(\mathbb{Z}/l^k\mathbb{Z})$, we have
    $$\sum_{\lambda_0}r(l^k, \lambda_0) = kq+O(\sqrt{q}).$$
    where the sum is taken over all \(\lambda_0\in\mathbb{P}^1(\mathbb{F}_q)\) at which \(E\) has good reduction.

\end{thm}
\begin{proof}
    We separate the proof into three cases, depending on the congruence of \(q \pmod{l}\), and proceed by induction on \(k\).

    \textbf{Case 1:} Assume \(q \not\equiv \pm 1 \pmod{l}\). We begin with the base case \(k = 1\). Suppose that \(E\) has good reduction at \(\lambda_0 \in \mathbb{P}^1(\mathbb{F}_q)\). Let \(\sigma_{q,\lambda_0}\in \operatorname{GL}_2(\mathbb{Z}/l\mathbb{Z})\) denote the matrix representing the action of the Frobenius \(q\)-map on \(E_{\lambda_0}[l]\). By the Weil pairing, we have $
    \zeta^{\det \sigma_{q,\lambda_0}} = \zeta^{\det \text{Frob}_{\lambda_0}} = \zeta^q \quad \text{for all } \zeta \in \mu_l$, so that
    \[
    \det \sigma_{q,\lambda_0} \equiv q \not\equiv \pm 1 \pmod{l}.
    \]

    In particular, if \(E_{\lambda_0}\) possesses a nonzero \(l\)-torsion point with \(x\)-coordinate in \(\mathbb{F}_q\), then \(\sigma_{q,\lambda_0}\) has an eigenvalue \(1\) or \(-1\), and $R^\pm(l, \lambda_0)$ are exactly the corresponding eigenspaces. Denote by \(\mathfrak{C}^+\) the conjugacy class of \(\begin{bmatrix} 1 & 0 \\ 0 & q \end{bmatrix}\) in \(\Gal(K(E[l])/K)\), and by \(\mathfrak{C}^-\) the conjugacy class of \(\begin{bmatrix} -1 & 0 \\ 0 & -q \end{bmatrix}\). We also use the notation $\lambda_0\in\mathfrak{C}^\pm$ to mean  \(\text{Frob}_{\lambda_0}\in \mathfrak{C}^{\pm}\). For each \(\lambda_0\in \mathfrak{C}^{\pm}\), we have
    \[
    r(l, \lambda_0) = \frac{l-1}{2}.
    \]

    Now, we count the size of the conjugacy class $\mathfrak{C}^{\pm}$. The stabilizer of such a diagonal matrix under conjugation consists of matrices of the form $\begin{bmatrix}a & 0 \\ 0 & d\end{bmatrix}$ with $ad\neq 0$. Also, $\Gal(K(E[l])/K)$ is a normal subgroup of $\text{GL}_2(\mathbb{Z}/l\mathbb{Z})$. This implies that the conjugacy class of such a matrix in $\Gal(K(E[l])/K)$ is the same as its conjugacy class in $\text{GL}_2(\mathbb{Z}/l\mathbb{Z})$. Therefore, we have 
    \[
    |\mathfrak{C}^\pm| = \frac{|\operatorname{GL}_2(\mathbb{Z}/l\mathbb{Z})|}{(l-1)^2}.
    \]
    By Theorem \ref{thm:Chebotarev}, the number of \(\lambda_0 \in \mathbb{P}^1(\mathbb{F}_q)\) with \(\text{Frob}_{\lambda_0} \in \mathfrak{C}^\pm\) is
    \[
    \frac{|\operatorname{GL}_2(\mathbb{Z}/l\mathbb{Z})|}{|\operatorname{SL}_2(\mathbb{Z}/l\mathbb{Z})|} \cdot \frac{1}{(l-1)^2} \, q + O(\sqrt{q})
    = \frac{1}{l-1}q + O(\sqrt{q}).
    \]
    Therefore,
    \[
    \sum_{\lambda_0} r(l, \lambda_0)
    = 2 \cdot \left(\frac{1}{l-1}q + O(\sqrt{q})\right) \cdot \frac{l-1}{2}
    = q + O(\sqrt{q}).
    \]

    Next, we proceed by induction by $k$. Assume the result holds for \(l^k\). We now prove it for \(l^{k+1}\). Notice that the roots of \(\Lambda_{l^k}(x, \lambda_0)\) are also roots of \(\Lambda_{l^{k+1}}(x, \lambda_0)\). Thus, it suffices to count the number of \(\mathbb{F}_q\)-roots of \(\Lambda_{l^{k+1}}(x, \lambda_0)\) that are not roots of \(\Lambda_{l^k}(x, \lambda_0)\), and show that the total contribution over all good \(\lambda_0\) is \(q + O(\sqrt{q})\). 

   Let \(\lambda_0\) be a good place such that \(E_{\lambda_0}\) has an \(l^{k+1}\)-torsion point which is not an \(l^k\)-torsion point with \(x\)-coordinate in \(\mathbb{F}_q\). Such a point is also an $l$-torsion point, and the previous argument implies that $\text{Frob}_{\lambda_0}$ has two distinct eigenvalues as a matrix in $\text{GL}_2(\mathbb{Z}/l\mathbb{Z})$. We can also view $\text{Frob}_{\lambda_0}$ as a matrix in $\text{GL}_2(\mathbb{Z}_l)$ via its action on the Tate module \(T_l(E_{\lambda_0})\). Hensel's lemma implies that $\text{Frob}_{\lambda_0}$ also has two distinct eigenvalues in $\mathbb{Z}_l$, and thus is diagonalizable as an $l$-adic matrix. Now, we view $\text{Frob}_{\lambda_0}$ as a matrix in $\text{GL}_2(\mathbb{Z}/l^{k+1}\mathbb{Z})$ via its action on $E_{\lambda_0}[l^{k+1}]$. We know that it is diagonalizable, and has canonical form $\pm\begin{bmatrix}1 & 0 \\ 0 &  q \end{bmatrix}$.

    Denote by \(\mathfrak{C}^\pm\) the corresponding conjugacy class in \(\Gal(K(E[l^{k+1}])/K) \subset \operatorname{GL}_2(\mathbb{Z}/l^{k+1}\mathbb{Z})\). For \(\lambda_0\in\mathbb{P}^1(\mathbb{F}_q)\) such that \(\text{Frob}_{\lambda_0}\in \mathfrak{C}^\pm\), the additional \(l^{k+1}\)-torsion points (beyond those coming from \(E[l^k]\)) contribute
    \[
    r(l^{k+1}, \lambda_0) - r(l^k, \lambda_0) = \frac{l^{k+1}-l^k}{2} = \frac{\varphi(l^{k+1})}{2}.
    \]
    Using the same counting method as in the base case, we find that
    \[
    |\mathfrak{C}^\pm| = \frac{|\operatorname{GL}_2(\mathbb{Z}/l^{k+1}\mathbb{Z})|}{|\operatorname{SL}_2(\mathbb{Z}/l^{k+1}\mathbb{Z})|} \cdot \frac{1}{\varphi(l^{k+1})^2} \, q + O(\sqrt{q})
    = \frac{1}{\varphi(l^{k+1})}q + O(\sqrt{q}).
    \]
    Hence,
    \[
    \sum_{\lambda_0} \Bigl(r(l^{k+1}, \lambda_0)-r(l^k, \lambda_0)\Bigr)
    = 2 \cdot \left(\frac{1}{\varphi(l^{k+1})}q + O(\sqrt{q})\right) \cdot \frac{\varphi(l^{k+1})}{2}
    = q + O(\sqrt{q}).
    \]
    Combining this with the inductive hypothesis, we obtain
    \[
    \sum_{\lambda_0} r(l^{k+1}, \lambda_0) = (k+1)q + O(\sqrt{q}).
    \]

\textbf{Case 2:} The proof for the case \( q \equiv -1 \pmod{l} \) is analogous to that of Case 1. In both cases, when \( \Lambda_l(x, \lambda_0) \) has \(\mathbb{F}_q\)-roots, the matrix \( \text{Frob}_{\lambda_0} \in \operatorname{GL}_2(\mathbb{Z}_l) \) is diagonalizable with two distinct eigenvalues.

\textbf{Case 3:} Now, suppose that \( q \equiv 1 \pmod{l} \). In this situation, if \( \Lambda_l(x, \lambda_0) \) has an \(\mathbb{F}_q\)-root, then \( \Frob_{\lambda_0} \) must have an eigenvalue \( \pm 1 \). Since 
\[
\det \Frob_{\lambda_0} \equiv q \equiv 1 \pmod{l},
\]
it follows that \( \Frob_{\lambda_0}\in\text{GL}_2(\mathbb{Z}/l\mathbb{Z}) \) has a repeated eigenvalue of either \( 1 \) or \( -1 \).

\emph{Subcase 3a:} If \( \Frob_{\lambda_0} \) is diagonalizable as a matrix in $\text{GL}_2(\mathbb{Z}/l\mathbb{Z})$, then its canonical form is \(\pm I_2\). Denote by \( \mathfrak{C}^\pm_\infty \) the conjugacy class of \(\pm I_2\) in \( \Gal(K(E[l])/K) \). For each \(\lambda_0 \in \mathfrak{C}^\pm_\infty\), we have $r(l, \lambda_0) = \frac{l^2-1}{2}$. Using Theorem \ref{thm:Chebotarev}, the number of \(\lambda_0 \in \mathbb{P}^1(\mathbb{F}_q)\) with \(\Frob_{\lambda_0}\in \mathfrak{C}^\pm_\infty\) is
\[
\frac{1}{|\operatorname{SL}_2(\mathbb{Z}/l\mathbb{Z})|}q + O(\sqrt{q}).
\]
Thus, the total contribution from these classes is
\[
\sum_{\lambda_0 \in \mathfrak{C}^\pm_\infty} r(l, \lambda_0)
= 2 \cdot \left(\frac{1}{|\operatorname{SL}_2(\mathbb{Z}/l\mathbb{Z})|}q + O(\sqrt{q})\right) \cdot \frac{l^2-1}{2}
= \frac{1}{l}q + O(\sqrt{q}).
\]

\emph{Subcase 3b:} If \( \Frob_{\lambda_0} \) is not diagonalizable, then its canonical form is $
\begin{bmatrix} 1 & 1 \\ 0 & 1 \end{bmatrix} $ or $\begin{bmatrix} -1 & 1 \\ 0 & -1 \end{bmatrix}$.
Denote these two conjugacy classes by \( \mathfrak{C}^\pm_0 \). The stabilizer of a matrix in \( \mathfrak{C}^\pm_0 \) consists of matrices of the form $
\begin{bmatrix} a & b \\ 0 & a \end{bmatrix}$ with $a\neq 0$. So there are 
\((l-1)l\) of them. For each \(\lambda_0 \in \mathfrak{C}^\pm_0\), we have $
r(l, \lambda_0) = \frac{l-1}{2}$. Thus, the total contribution from these classes is
\[
\sum_{\lambda_0 \in \mathfrak{C}^\pm_0} r(l,\lambda_0)
= 2 \cdot \left(\frac{|\operatorname{GL}_2(\mathbb{Z}/l\mathbb{Z})|}{|\operatorname{SL}_2(\mathbb{Z}/l\mathbb{Z})|} \cdot \frac{1}{(l-1)l} \, q + O(\sqrt{q})\right) \cdot \frac{l-1}{2}
= \frac{l-1}{l}q + O(\sqrt{q}).
\]

Therefore, in Case 3 the total number of \(\mathbb{F}_q\)-roots of \(\Lambda_l(x, \lambda_0)\) over all good \(\lambda_0\) is
\[
\sum_{\lambda_0} r(l, \lambda_0)
= \sum_{\lambda_0 \in \mathfrak{C}^\pm_\infty} r(l, \lambda_0)
+ \sum_{\lambda_0 \in \mathfrak{C}^\pm_0} r(l, \lambda_0)
= \frac{1}{l}q + \frac{l-1}{l}q + O(\sqrt{q})
= q + O(\sqrt{q}).
\]

We then proceed by induction on \( k \) to prove the statement for \( l^{k+1} \). We will show that the number of new \(\mathbb{F}_q\)-roots to \(\Lambda_{l^{k+1}}(x, \lambda_0)\) that are not roots of \(\Lambda_{l^k}(x, \lambda_0)\) is \( q + O(\sqrt{q}) \) when summed over all good \(\lambda_0\). 

If \(\Lambda_{l^{k+1}}(x, \lambda_0)\) has an \(\mathbb{F}_q\)-root that is not a root of \(\Lambda_{l^k}(x, \lambda_0)\), then there exists some \(P \in E_{\lambda_0}[l^{k+1}] \setminus E_{\lambda_0}[l^k]\) such that \(\Frob_{\lambda_0}(P) = \pm P\). If \(q \not\equiv 1 \pmod{l^{k+1}}\), the characteristic polynomial of \(\Frob_{\lambda_0}\) has two distinct roots modulo \(l^{k+1}\). By Hensel's lemma, \(\Frob_{\lambda_0}\) is diagonalizable in \(\mathbb{Z}_l\) and hence in \(\mathbb{Z}/l^{k+1}\mathbb{Z}\); the argument then proceeds as in the base case.

Now suppose \(q \equiv 1 \pmod{l^{k+1}}\). For simplicity, we restrict our attention to  those \(P \in E_{\lambda_0}[l^{k+1}] \setminus E_{\lambda_0}[l^k]\) satisfying \(\Frob_{\lambda_0}(P) = P\). Pick a basis \(\{P, Q\}\) for \(E_{\lambda_0}[l^{k+1}]\). By a change of basis to $\{P, cQ\}$ for some suitable $c$, we may assume that $\Frob_{\lambda_0} = \begin{bmatrix} 1 & b \\ 0 & 1 \end{bmatrix} \in \text{GL}_2(\mathbb{Z}/l^{k+1}\mathbb{Z})$. 

We pick a representative of $b\in\mathbb{Z}/l^{k+1}\mathbb{Z}$ as an integer between $0\leq b\leq l^{k+1}-1$. If $b = 0$, then $\text{Frob}_{\lambda_0}$ has the canonical form $\begin{bmatrix}
    1 & 0 \\ 0 & 1
\end{bmatrix}$, and we denote by $\mathfrak{C}_\infty^+$ the conjugacy class of this element in $\Gal(K(E[l^{k+1}])/K)$. 

If $b\neq 0$, we write $b = l^vx$ for some $0\leq v \leq k$ and $x$ coprime to $l$. We preform a change of basis by replacing $\{P, Q\}$ by $\{xP, Q\}$, and we see that $\Frob_{\lambda_0}$ has the canonical form $\begin{bmatrix}
    1 & l^v \\ 0 & 1
\end{bmatrix}$. We denote by $\mathfrak{C}_v^+$ the conjugacy of this element in $\Gal(K(E[l^{k+1}])/K)$. 

Next, we compute the size of those conjugacy classes by counting the size of their stabilizers in $\text{GL}_2(\mathbb{Z}/l^{k+1}\mathbb{Z})$. The stabilizer of $\begin{bmatrix}
    1 & l^v \\ 0 & 1
\end{bmatrix}$ consists of matrices of the form $\begin{bmatrix}a & b \\ c & d\end{bmatrix}$ where $(a-d)l^{v} = 0$ and $c l^{v} = 0$. It has size $\varphi(l^{k+1})l^{v}l^{k+1} l^{v}$, where each factor represents the number of choices for $a,b,c,d$ respectively. In particular, for $0\leq v\leq k$, we have
$$|\mathfrak{C}^+_v| = \frac{|\text{GL}_2(\mathbb{Z}/l^{k+1}\mathbb{Z})|}{\varphi(l^{k+1})l^{2v+k+1}}. $$
For $v = \infty$, we have 
$$|\mathfrak{C}^+_\infty| = |\text{GL}_2(\mathbb{Z}/l^{k+1}\mathbb{Z})| = (l^{2(k+1)} - l^{2(k+1)-1})(l^{2(k+1)} - l^{2(k+1)-2}). $$

Then, we compute $r^+(l^{k+1}, \lambda_0)-r^+(l^{k}, \lambda_0)$ for $\lambda_0$ in each of those conjugacy classes. We choose a basis $\{P,Q\}$ of $E_{\lambda_0}[l^{k+1}]$ such that $\Frob_{\lambda_0}$ has the canonical form described above. For $0\leq v\leq k$, and $\lambda_0\in\mathfrak{C}_v^+$ we have 
$$R^+(l^{k+1}, \lambda_0)\setminus R^+(l^k, \lambda_0) = \left\{xP+yl^{k+1-v}Q: x\in (\mathbb{Z}/l^{k+1}\mathbb{Z})^\times, 0\leq y< l^{v}\right\},$$
and thus 
$$r^+(l^{k+1}, \lambda_0) - r^+(l^{k}, \lambda_0) = \frac{l^{v}\varphi(l^{k+1})}{2}.$$
Therefore, the contribution from the conjugacy class \(\mathfrak{C}^+_v\) is
\begin{align*}
\sum_{\lambda_0 \in \mathfrak{C}^+_v} \Bigl(r^+(l^{k+1}, \lambda_0) - r^+(l^k, \lambda_0)\Bigr)
& =  \left(\frac{|\text{GL}_2(\mathbb{Z}/l^{k+1}\mathbb{Z})|}{|\operatorname{SL}_2(\mathbb{Z}/l^{k+1}\mathbb{Z})|}\cdot \frac{1}{\varphi(l^{k+1})l^{2v+k+1}}q + O(\sqrt q)\right) \cdot \frac{l^v\varphi(l^{k+1})}{2} \\
&= \frac{1}{2}\frac{\varphi(l^{k+1})}{l^{k+1}l^v}q + O(\sqrt{q}).
\end{align*}
Similarly, the contribution from the conjugacy class $\mathfrak{C}_\infty^+$ is 
\begin{align*}\sum_{\lambda_0\in\mathfrak{C}_\infty^+}\Bigl(r^+(l^{k+1}, \lambda_0) - r^+(l^k, \lambda_0)\Bigr) &= \left(\frac{1}{|\operatorname{SL}_2(\mathbb{Z}/l^{k+1}\mathbb{Z})|}q + O(\sqrt{q})\right)\cdot \frac{l^{2(k+1)} - l^{2k}}{2}  \\
&= \frac{1}{2}\frac{1}{l^{k+1}}q+O(\sqrt{q}).\end{align*}

Summing over all conjugacy classes, we deduce that
\begin{align*}
    \sum_{\lambda_0}\Bigl(r^+(l^{k+1}, \lambda_0) - r^+(l^k, \lambda_0)\Bigr) & = \sum_{\lambda_0\in\mathfrak{C}_\infty^+} \Bigl(r^+(l^{k+1}, \lambda_0) - r^+(l^k, \lambda_0)\Bigr) + \sum_{v = 0}^k\sum_{\lambda_0\in \mathfrak{C}^+_v} \Bigl(r^+(l^{k+1}, \lambda_0) - r^+(l^k, \lambda_0)\Bigr)  \\
    &= \frac{1}{2}\frac{1}{l^{k+1}}q + \sum_{v=0}^k  \left(\frac{1}{2}\frac{\varphi(l^{k+1})}{l^{k+1}l^v}q + O(\sqrt{q}\right) +  O(\sqrt{q}) \\
    &= \frac{1}{2}q\left(\frac{1}{l^{k+1}} + \frac{\varphi(l^{k+1})}{l^{k+1}}\frac{l-l^{-k}}{l-1} \right) + O(\sqrt q) \\
    &= \frac{1}{2}q+O(\sqrt{q}).
\end{align*}

A similar computation shows that the contribution from the points $P$ such that $\text{Frob}_{\lambda_0}(P) = -P$ is also \( \frac{1}{2}q + O(\sqrt{q}) \). Hence,
\[
\sum_{\lambda_0} \Bigl( r(l^{k+1}, \lambda_0) - r(l^k, \lambda_0) \Bigr) = q + O(\sqrt{q}).
\]
This completes the induction and the proof.

\end{proof}

\subsection{Main theorem on the $n$-torsion points of a family of elliptic curves} 
We now generalize the results of the previous section to all odd integers \(n\) coprime to \(q\). In fact, the following lemma implies that if \(m\) and \(n\) are coprime odd integers, then
\[
r^\pm(mn, \lambda_0) - r^\pm(m, \lambda_0) - r^\pm(n, \lambda_0) = 2\,r^\pm(m, \lambda_0)\,r^\pm(n, \lambda_0).
\]
Using this result together with Theorem \ref{thm: main l^k}, we are now in a position to prove Theorem \ref{thm:main1}, our main theorem on the \(n\)-torsion points of a family of elliptic curves.

\begin{lem}
    Let \(E_{\lambda_0}\) be an elliptic curve over the finite field \(\mathbb{F}_q\), and $\sigma_{q, \lambda_0}$ be the Frobenious $q$-map. Let \(m,n\) be coprime integers, and let \(P\in E_{\lambda_0}[n]\) and \(Q\in E_{\lambda_0}[m]\). Then, \(x(P+Q)\in \mathbb{F}_q\) if and only if either 
    \[
    \sigma_{q, \lambda_0}(P)=P \text{ and } \sigma_{q, \lambda_0}(Q)=Q,
    \]
    or
    \[
    \sigma_{q, \lambda_0}(P)=-P \text{ and } \sigma_{q, \lambda_0}(Q)=-Q.
    \]
\end{lem}
\begin{proof}
    If \(x(P+Q)\in \mathbb{F}_q\), then \(\sigma_{\lambda_0,q}(P+Q)=P+Q\) or \(\sigma_{\lambda_0,q}(P+Q) = -(P+Q)\). Since \(m\) and \(n\) are coprime, the decomposition of an \(mn\)-torsion point into an \(m\)-torsion point and an \(n\)-torsion point is unique. Hence, either \(\sigma_{\lambda_0,q}(P)=P\) and \(\sigma_{\lambda_0,q}(Q)=Q\) or \(\sigma_{\lambda_0,q}(P)=-P\) and \(\sigma_{\lambda_0,q}(Q)=-Q\).
\end{proof}

It follows that
\[
R^\pm(mn, \lambda_0) \setminus \bigl(R^\pm(m, \lambda_0) \cup R^\pm(n, \lambda_0)\bigr) = \left\{P+Q: P \in R^\pm(m, \lambda_0)\setminus\{O\}, Q\in R^\pm(n, \lambda_0)\setminus\{O\} \right\}
\]
Using the relation $
r^\pm(n, \lambda_0) = \frac{|R^\pm(n, \lambda_0)|-1}{2}$,
we deduce that
\[
r^\pm(mn, \lambda_0) - r^\pm(m, \lambda_0) - r^\pm(n, \lambda_0) = 2\,r^\pm(m, \lambda_0)\,r^\pm(n, \lambda_0).
\]

For the convenience of the reader, we now restate Theorem \ref{thm:main1} before proving it.

\begin{thm*}(Main result on the $n$-torsion points of a family of elliptic curves) \par
    Let $\mathbb{F}_q$ be a finite field with characteristic greater than $3$. Let $E$ be a non-isotrivial elliptic curve defined over the function field $K = \mathbb{F}_q(\lambda)$. If $n$ is an odd integer coprime to $q$, and the geometric Galois group of $K(E[n])/K$ is $\operatorname{SL}_2(\mathbb{Z}/n\mathbb{Z})$, we have
    $$\sum_{\lambda_0}r(n,\lambda_0) = (d(n)-1)q+O(\sqrt{q})$$
    where the sum is over all $\lambda_0\in\mathbb{P}^1(\mathbb{F}_q)$ such that $E$ has a good reduction at $\lambda_0$. 
\end{thm*}

\begin{proof}
    We prove the theorem by induction on \(k\), the number of distinct prime factors of \(n\). The base case \(k=1\) is established in Theorem \ref{thm: main l^k}. Suppose the statement holds for $n' = \prod_{i=1}^k l_i^{a_i}$, and consider $n = n'\, l_{k+1}^{a_{k+1}}$,
    where \(l_{k+1}\) is coprime to \(n'\).

    Since \(n', l_{k+1}\) are coprime, we have the injection 
    \[
    \Gal\bigl(K(E[n])/K\bigr) \hookrightarrow \Gal\bigl(K(E[n'])/K\bigr) \times \Gal\bigl(K(E[l_{k+1}^{a_{k+1}}])/K\bigr).
    \]
    which is compatible with $\operatorname{GL}_2(\mathbb{Z}/n\mathbb{Z})\simeq\text{GL}_2(\mathbb{Z}/n'\mathbb{Z})\times\text{GL}_2(\mathbb{Z}/l_{k+1}^{a_{k+1}}\mathbb{Z})$. By using the assumption that the geometric Galois group of $K(E[n])/K$ is $\text{SL}_2(\mathbb{Z}/n\mathbb{Z})$, and comparing the cardinality of both sides, we conclude that 
    \[
    \Gal\bigl(K(E[n])/K\bigr) \simeq \Gal\bigl(K(E[n'])/K\bigr) \times \Gal\bigl(K(E[l_{k+1}^{a_{k+1}}])/K\bigr).
    \]

    Let \(\mathfrak{C}^+_i\) be the conjugacy classes in \(\Gal(K(E[n'])/K)\) for which \(r^+(n', \lambda_0) > 0\), and let \(\mathfrak{D}^+_j\) be those in \(\Gal(K(E[l_{k+1}^{a_{k+1}}])/K)\) for which \(r^+(l_{k+1}^{a_{k+1}}, \lambda_0) > 0\). Then, we may view the Cartesian product \(\mathfrak{C}^+_i \times \mathfrak{D}^+_j\) as a conjugacy class in \(\Gal(K(E[n])/K)\). For \(\lambda_0\in\mathbb{P}^1(\mathbb{F}_q)\) with $\Frob_{\lambda_0}\in \mathfrak{C}_i^+\times\mathfrak{D}_j^+$, we have 
    \[
    r^+(n, \lambda_0) - r^+(n', \lambda_0) - r^+(l_{k+1}^{a_{k+1}}, \lambda_0) = 2\,r^+(n', \lambda_0)\,r^+(l_{k+1}^{a_{k+1}}, \lambda_0).
    \]

    Summing over all good \(\lambda_0\) gives
    \begin{align*}
    & \sum_{\lambda_0}\Bigl( r^+(n, \lambda_0) - r^+(n', \lambda_0) - r^+(l_{k+1}^{a_{k+1}}, \lambda_0) \Bigr) \\
    &= \sum_{i}\sum_{j}\left( \frac{|\mathfrak{C}^+_i \times \mathfrak{D}^+_j|}{|\operatorname{SL}_2(\mathbb{Z}/n\mathbb{Z})|}q + O(\sqrt{q}) \right) \cdot \Bigl(2\,r^+(n', \lambda_0)\,r^+(l_{k+1}^{a_{k+1}}, \lambda_0)\Bigr) \\
    &= 2\,\left[\sum_i \left(\frac{|\mathfrak{C}^+_i|}{|\operatorname{SL}_2(\mathbb{Z}/n' \mathbb{Z})|}q \right)r^+(n', \lambda_0)\right] \cdot \left[\sum_j \left(\frac{|\mathfrak{D}^+_j|}{|\operatorname{SL}_2(\mathbb{Z}/l_{k+1}^{a_{k+1}} \mathbb{Z})|}q\right)r^+(l_{k+1}^{a_{k+1}}, \lambda_0)\right] + O(\sqrt{q}) \\
     & = 2\left(\frac{\prod_1^k (a_i+1)-1}{2}\right)\left( \frac{a_{k+1}}{2}\right)q + O(\sqrt q)\\
    &= \frac{1}{2}\left(a_{k+1} \prod_{i=1}^k (a_i+1) - a_{k+1}\right)q + O(\sqrt{q}).
    \end{align*}

    The calculation for the points $P\in E_{\lambda_0}[n]\setminus\{O\}$ with $\text{Frob}_{\lambda_0}(P) = -P$ is the same, and we get 
    \[
    \sum_{\lambda_0}\Bigl( r(n, \lambda_0) - r(n', \lambda_0) - r(l_{k+1}^{a_{k+1}}, \lambda_0) \Bigr)
    = \Bigl(a_{k+1} \prod_{i=1}^k (a_i+1) - a_{k+1}\Bigr)q + O(\sqrt{q}).
    \]
   
    Combining these, we obtain
    \begin{align*}
        \sum_{\lambda_0} r(n, \lambda_0) &= \sum_{\lambda_0}r(n', \lambda_0) + \sum_{\lambda_0}r(l_{k+1}^{a_{k+1}}, \lambda_0) + \sum_{\lambda_0} \Bigl( r(n, \lambda_0) - r(n', \lambda_0) - r(l_{k+1}^{a_{k+1}}, \lambda_0) \Bigr) \\
        &= \Bigl(\prod_{i=1}^k(a_i+1) - 1\Bigr)q + a_{k+1}q + \Bigl(a_{k+1} \prod_{i=1}^k (a_i+1) - a_{k+1}\Bigr)q + O(\sqrt{q})\\
        &= \Bigl( \prod_{i=1}^{k+1} (a_i+1) - 1 \Bigr) q + O(\sqrt{q}) \\
        &= (d(n)-1)q + O(\sqrt{q}).
    \end{align*}
\end{proof}

\section{Conjectures for even $n$ and computational data}
We state a conjecture on the density of pairs of smooth conics satisfying the Poncelet $n$-gon condition including the even integers. 
\subsection{Conjectures}
\begin{conj}
    Let $\mathbb{F}_q$ be a finite field with characteristic greater than $3$, and $n$ be an integer coprime to $q$. We have
    $$\frac{|\Gamma_n|}{|\Psi|} = \frac{d'(n)}{q}+O(q^{-3/2}), $$
    where \(d'(n)\) denotes the number of divisors of \(n\) other than \(1\) and \(2\).
\end{conj}

Further, if we assume the following results on the geometric Galois group of $K(E[2^m])/K$, we can prove the conjecture following the same methods in this paper. We sketch some results based on this assumption, and provide the computational data at the end of this section. 

\subsubsection{Pencil intersection type: \((1,1,1,1)\) or \((2,2)\)}\label{conj_(1,1,1,1)_pencil}

\begin{itemize}
    \item Density of pencil: $\frac{1}{24} + \frac{1}{8} = \frac{1}{6}$.
    \item Corresponding elliptic curve in Theorem \ref{thm:3.2}: $E: y^2 = x(x-1)(x-\lambda)$.
    \item Assumption on the geometric Galois group: $\ker\Bigl(\operatorname{SL}_2(\mathbb{Z}/2^m\mathbb{Z}) \xrightarrow{\mod 2} \operatorname{SL}_2(\mathbb{Z}/2\mathbb{Z})\Bigr)$.
    \item Total number of $\mathbb{F}_q$-roots of the $2^m$-th division polynomial: $\sum_{\lambda_0} r(2^m, \lambda_0) = 3mq + O(\sqrt{q})$.
    \item Density of pairs of smooth conics satisfying the Poncelet $2^m$-gon condition: $\frac{(3m-3)}{q} + O(q^{-3/2})$.
\end{itemize}
Recall Theorem \ref{thm:3.2} implies that we need to remove the $2$-torsion points when counting the number of pairs of conics with a Poncelet $n$-gon. There are three $2$-torsion points on $E_{\lambda_0}$ for each $\lambda_0$. This is why we subtract $3$ when computing the density of pairs of smooth conics satisfying the Poncelet $2^m$-gon condition. 

\subsubsection{Pencil intersection type: \((2,1,1)\) or \((4)\)}

\begin{itemize}
    \item Density of pencil: $\frac{1}{4} + \frac{1}{4} = \frac{1}{2}$.
    \item Corresponding elliptic curve in Theorem \ref{thm:3.2}: $E: y^2 = (x-\lambda)(x^2-b)$ where $b$ is not a square in $\mathbb{F}_q$. .
    \item Assumption on the geometric Galois group: pre-image of $\left\{I, \begin{bmatrix}
    0 & 1 \\ 1 & 0 \end{bmatrix}\right\}$ under $\operatorname{SL}_2(\mathbb{Z}/2^m\mathbb{Z}) \to\operatorname{SL}_2(\mathbb{Z}/2\mathbb{Z})$.
    \item Total number of $\mathbb{F}_q$-roots of the $2^m$-th division polynomial: $\sum_{\lambda_0} r(2^m, \lambda_0) = mq + O(\sqrt{q})$.
    \item Density of pairs of smooth conics satisfying the Poncelet $2^m$-gon condition: $\frac{(m-1)}{q} + O(q^{-3/2}) $.
\end{itemize}

\subsubsection{Pencil intersection type: $(3, 1)$}
\begin{itemize}
    \item Density of pencil: $\frac{1}{4} + \frac{1}{4} = \frac{1}{3}$.
    \item Corresponding elliptic curve in Theorem \ref{thm:3.2}: $E: y^2 = \left(x - \frac{\alpha}{\alpha + \lambda}\right)\left(x - \frac{\alpha'}{\alpha' + \lambda}\right)\left(x - \frac{\alpha''}{\alpha'' + \lambda}\right)$.
    \item Total number of $\mathbb{F}_q$-roots of the $2^m$-th division polynomial: $\sum_{\lambda_0} r(2^m, \lambda_0) = 0$.
    \item Density of pairs of smooth conics satisfying the Poncelet $2^m$-gon condition: $0$.
\end{itemize}

For any $\lambda_0\in\mathbb{P}^1(\mathbb{F}_q)$, the elliptic curve $E_{\lambda_0}$ has no $2$-torsion points with $x$-coordinate defined over $\mathbb{F}_q$, corresponding to the fact that there is no singular conics in this pencil defined over $\mathbb{F}_q$. Therefore, there are no $2^m$-torsion points with $x$-coordinate defined over $\mathbb{F}_q$ for any $m\geq 1$. 

If we take the weighted sum over all pencils, we see that the density of pairs of smooth conics satisfying the Poncelet $2^m$-gon condition is 
$$\frac{1}{6}\frac{3m-3}{q} + \frac{1}{2}\frac{m-1}{q} + O(q^{-3/2})= \frac{m-1}{q}+O(q^{-3/2}) = \frac{d'(2^m)}{q}+O(q^{-3/2})$$.

All of the results above are consistent with our computation in Section \ref{tetragon}. Also, for the Poncelet $8$-gon, the result above suggests that the density of pairs of smooth conics satisfying the Poncelet $n$-gon condition is $\frac{6}{q} + O(q^{-3/2})$ in the pencil of intersection type $(1,1,1,1)$. This also matches the conjecture of Chipalkatti for $n=8$ \cite[Conjecture 4.1]{Chi}. 

\subsection{Computational data}
Our computational data supports the results above, although we are not able to prove our assumption on the geometric Galois group of \(K(E[2^m])/K\).

We performed our computations using SageMath. In our experiments, we fix a finite field \(\mathbb{F}_p\) and consider the Legendre curve $y^2 = x(x-1)(x-\lambda)$. For each \(n\), we compute the sum of the \(\mathbb{F}_q\)-roots of the \(n\)th division polynomial of the reduced curve \(E_{\lambda_0}\) as \(\lambda_0\) ranges over \(\mathbb{F}_q \setminus \{0,1\}\).

\medskip

\textbf{\(n = 4\):}

\begin{center}
\begin{tabular}{c|c}
\(p\) & (Sum of \(\mathbb{F}_q\)-roots to 4th-division polynomial)/\(p\) \\\hline
1487 & 5.98991 \\
1489 & 5.98993 \\ 
1493 & 5.98995 \\ 
1499 & 5.98999 \\ 
1511 & 5.99007 \\ 
1523 & 5.99015 \\ 
1531 & 5.99020 \\
1543 & 5.99028 \\
\end{tabular}
\end{center}

\medskip

\textbf{\(n = 8\):}

\begin{center}
\begin{tabular}{c|c}
\(p\) & (Sum of \(\mathbb{F}_q\)-roots to 8th-division polynomial)/\(p\) \\\hline
1993 & 8.97893 \\
1997 & 8.98498 \\
1999 & 8.98199 \\
2003 & 8.98802 \\
2011 & 8.98807 \\
2017 & 8.97918 \\
2027 & 8.98816 \\
2029 & 8.98521 \\
2039 & 8.98234 \\
2053 & 8.98539 \\
2063 & 8.98255 \\
2069 & 8.98550 \\
\end{tabular}
\end{center}

\section*{Acknowledgment}
I would like to thank Aaron Landesman for convincing me that the geometric Galois group of \(K(E[n])/K\) is \(\operatorname{SL}_2(\mathbb{Z}/n\mathbb{Z})\) when \(E\) is the Legendre curve over \(K = \mathbb{F}_q(\lambda)\), and Chris Hall for providing a reference for the proof of this statement. I am also grateful to my advisor, Nathan Kaplan, for introducing me to this intriguing problem and for the many valuable discussions.

\end{document}